\documentclass[reqno]{amsart}

\usepackage[top=1in, bottom=1in, left=1in, right=1in]{geometry}

\usepackage{amssymb}
\usepackage{}
\usepackage{amscd}

\newtheorem{theorem}{Theorem}[section]
\newtheorem{lemma}[theorem]{Lemma}

\theoremstyle{definition}
\newtheorem{definition}[theorem]{Definition}

\newtheorem{prop}[theorem]{Proposition}

\newtheorem{remark}[theorem]{Remark}

\renewcommand{\d}{\mbox{\rm d}}

\newcommand{\Sp}{\mbox{\rm Sp}}

\newcommand{\Tr}{\mbox{\rm Tr}}
\numberwithin{equation}{section}

\begin{document}
\pagestyle{headings}
\title{Differential operators for Siegel-Jacobi forms}

\author{Jiong Yang and Linsheng Yin}
\address{Department of Mathematical Science, Tsinghua University, Beijing, P. R. China 100084}
\email{yangjiong10@mails.tsinghua.edu.cn, lsyin@math.tsinghua.edu.cn.}
\thanks{The authors were supported  by NSFC NO.11271212.}





\begin{abstract}
For any positive integers $n$ and $m$, $\mathbb{H}_{n,m}:=\mathbb{H}_n\times\mathbb{C}^{(m,n)}$ is called the Siegel-Jacobi space,
with the Jacobi group acting on it. The Jacobi forms are defined on this space. In this article we compute the Chern
 connection of the Siegel-Jacobi space and use it to obtain derivations of Jacobi forms. Using these results, we constructed a series of invariant differential operators for Siegel-Jacobi forms. Also two kinds of Maass-Shimura type differential operators for $\mathbb{H}_{n,m}$ are obtained.
\end{abstract}

\maketitle
\section{Introduction}

We first recall the concept of Jacobi forms and their differential operators. For given fixed positive integers $n,m$, let $\mathbb{H}_n:=\{Z\in M_{n,n}(\mathbb{C})|Z=Z^t, Im(Z)>0\}$ be the Siegel upper half plane of degree $n$, and $\mathbb{H}_{n,m}:=\mathbb{H}_n\times\mathbb{C}^{(m,n)}$, the Siegel-Jacobi space. An element of $\mathbb{H}_{n,m}$ can be written as $(Z,W)$ with $Z=Z^t=(z_{ij})\in M_{n,n},W=(w_{rs})\in M_{m,n}. $ Let $Y$ and $V$ be the imaginary part of the matrix $Z$ and $W$ respectively.

Let $\Sp(n,\mathbb{R})$ be the symplectic group of degree $n$. The Heisenberg group $H_\mathbb{R}^{(n,m)}$ is defined to be the set

\[H_\mathbb{R}^{(n,m)}:=\left\{(\lambda,\mu;\kappa)\  | \ \lambda,\mu\in\mathbb{R}^{(m,n)},\kappa\in\mathbb{R}^{(m,m)},\kappa+\mu\lambda^t \ \text{symmetric}\right\}.\]

 endowed with the multiplication law:

\[(\lambda,\mu;\kappa)\circ(\lambda',\mu';\kappa')=(\lambda+\lambda',\mu+\mu',\kappa+\kappa'+\lambda^t\mu'-\mu^t\lambda').\]

The Jacobi group of degree $n$, index $m$ is defined to be $G^J:=\Sp(n,\mathbb{R})\ltimes H_\mathbb{R}^{(n,m)}$, endowed with the following multiplication law
\[\left(\gamma,(\lambda,\mu;\kappa)\right)\cdot\left(\gamma',(\lambda',\mu';\kappa')\right)
=\left(\gamma \gamma',(\widetilde{\lambda}+\lambda',\widetilde{\mu}+\mu';\kappa+\kappa'+\widetilde{\lambda}\mu'^t-\widetilde{\mu}\lambda'^t)\right)\]

where $\gamma,\gamma'\in \Sp(n,\mathbb{R});(\lambda,\mu;\kappa),(\lambda',\mu';\kappa')\in H_{\mathbb{R}}^{(n,m)}$and $(\widetilde{\lambda},\widetilde{\mu})=(\lambda,\mu)\Gamma'$. The Jacobi group $G^J$ acts canonically on $\mathbb{H}_{n,m}$ by
\begin{eqnarray}\label{equation:action}
\left(\gamma,(\lambda,\mu;\kappa)\right)\cdot(Z,W)=\left(\gamma\cdot Z,(W+\lambda Z+\mu)(CZ+D)^{-1}\right),
\end{eqnarray}
where $\gamma\cdot Z$ means the usual action of an element $\gamma=\left(
     \begin{array}{cc}
       A & B \\
       C & D \\
     \end{array}
   \right)\in\Sp(n,\mathbb{R})$ on the Siegel upper half plane by $\gamma\cdot Z=(AZ+B)(CZ+D)^{-1}$. See for example \cite{Maass} for more details.

To define Jacobi forms, we need some discrete subgroups of $G^J$. Denote the discrete subgroup $\Sp(n,\mathbb{Z})\ltimes \mathbb{Z}^2$ of $G^J$ by $\Gamma^J$. We can also consider more general discrete subgroups, but now we concentrate on the $\Sp(n,\mathbb{Z})\ltimes \mathbb{Z}^2$ case.

We give the following formal definition of Jacobi forms of general degree.
\begin{definition}
Let $M$ be a positive definite half integer $m\times m$ matrix, a ( holomorphic ) Jacobi form $f$ of weight $k$ and index $M$, is a ( holomorphic ) function on $\mathbb{H}_{n,m}$, which satisfies the translation law of :
\begin{equation}
f(g(Z,W))=\det(CZ+D)^k\cdot e^{2\pi \sqrt{-1} \cdot\Tr\left\{(MW(CZ+D)^{-1}CW^t-M(\lambda Z\lambda^t+2\lambda W^t-\mu\lambda^t))\right\}}\cdot f(Z,W),
\end{equation}
for $g=\left(\left(
     \begin{array}{cc}
       A & B \\
       C & D \\
     \end{array}
   \right),(\lambda,\mu,\kappa)\right)\in \Gamma^J$.
\end{definition}

This is the definition of Ziegler in\cite{Ziegler}. In this definition of Jacobi forms, we did not ask for any growth conditions, such as assuming they are holomorphic or eigenfunctions of some differential operator. Since our purpose is concentrated on differential operators, although these conditions are needed in practise, now we do not assume this. The set of Jacobi forms of weight $k$ and index $M$ is denoted by $J_{k,M}$, and the subset of holomorphic Jacobi forms are denoted by $J_{k,M}^{hol}$.

Jacobi forms of degree 1 are introduced and studied by Eicher and Zagier in \cite{Zagier}. They are related to both modular forms and degree 2 Siegel modular forms and are useful in many aspects of number theory such as lifting problems. The non-holomorphic differential operators on Jacobi forms are well studied in \cite{Berndt Schmidt}. There are several important differential operators for Jacobi forms of degree 1, including two raising operators and two lowering operators,  which are
\begin{eqnarray*}
\nonumber Y_+f&=&  \sqrt{-1}\left(\frac{\partial f}{\partial w}+\frac{v}{y}4\pi \sqrt{-1}Mf\right),\\
\nonumber Y_-f&=&  y\frac{\partial f}{\partial \overline{w}},\\
\nonumber X_+f&=& \sqrt{-1} \left(2\frac{\partial f}{\partial z}+2\frac{v}{y}\frac{\partial f}{\partial w}+4\pi\sqrt{-1}M\frac{v^2}{y^2}f+\frac{k}{y}f\right), \\
\nonumber X_-f&=&  2\sqrt{-1}y\left(y\frac{\partial f}{\partial \overline{z}}+v\frac{\partial f}{\partial \overline{w}}\right).
\end{eqnarray*}

Here $(z,w)$ is the local coordinate of $\mathbb{H}_{1,1}$, and $y$ and $v$ are their imaginary part respectively. Also there is a non-holomorphic heat operator which is related to the operator $\frac{\partial}{\partial z}-\frac{\sqrt{-1}k}{y}f$ in the correspondence of Jacobi forms and half integer modular forms. The heat operator is defined as (denoted as $D_+$ in \cite{Berndt Schmidt})
\[Lf:=8\pi M\sqrt{-1}\frac{\partial f}{\partial z}-\frac{\partial^2f}{\partial w^2}-\frac{2M\pi-4M\pi k}{y}f.\]
These operators can help us understand Jacobi forms better and we generalize them to higher degree cases in section \ref{section:differential}. Although they are not linear, we will see later that this will help us construct invariant differential operators and Maass-Shimura operators.

 We will consider their generalizations to general degree.
In the paper \cite{Yang Yin}, Yang and Yin investigated  the derivative operators of Siegel modular forms. Similar to this work, we get the following differential operators.

\begin{theorem}[See Theorem \ref{theorem:R}]\label{theorem:introR}
For any $f\in J_{k,M}$, we have
\begin{itemize}
    \item[(a)] If $n=m$,  $R_1(f)=\det\left(\frac{\partial f}{\partial W}+4\pi \sqrt{-1}(Y^{-1}V^tM)f\right)\in J_{nk+1,nM}$;
  \item[(b)]If $n=m$,  $L_1(f)=\det\left(\frac{\partial f}{\partial \overline{W}}Y\right)$ is in $J_{nk-1,nM}$.
   \item[(c)]$R_2(f)=\det\left(\frac{\partial f}{\partial Z}-\frac{\sqrt{-1}k}{2}fY^{-1}+2\pi\sqrt{-1}Y^{-1}V^tMVY^{-1}f+\frac{1}{2}\frac{\partial f}{\partial W}VY^{-1}+\frac{1}{2}Y^{-1}V^t\frac{\partial f}{\partial W}^t\right)\in J_{nk+2,nM}$
   \item[(d)]$L_2(f)=\det\left(\frac{\partial f}{\partial \bar{Z}}Y^2+\frac{1}{2}\frac{\partial f}{\partial \overline{W}}VY+\frac{1}{2}YV^t\frac{\partial f}{\partial \overline{W}}^t\right)\in J_{nk-2,nM}$
\end{itemize}
\end{theorem}
They are generalizations of $Y_+,Y_-,X_+$ and $X_-$ respectively. The proof of this theorem is contained in section \ref{section:differential1} and \ref{section:differential2}, using the following Theorem \ref{theorem:introconnection} of the Chern connections on the Siegel-Jacobi space. The authors are grateful to the reviewers for the advise of using Chern connections instead of the Levi-Civita connection in the following theorem.

\begin{theorem}[See Theorem \ref{theorem:connection}]\label{theorem:introconnection}
Let $\mathbb{D}$ be the Chern connection on the Hermitian manifold $\mathbb{H}_{n,m}$ associated to the invariant metric, then $\mathbb{D}$ satisfies
\begin{eqnarray}
   \mathbb{D}(\d Z)&=&-\frac{\sqrt{-1}B}{2A}(\d Z,\d W^t)\left(
     \begin{array}{cc}
      2\frac{A}{B}Y^{-1}+Y^{-1}V^tVY^{-1} & -Y^{-1}V^t \\
       -VY^{-1} & I\\
     \end{array}
   \right)\left(
     \begin{array}{c}
       \d Z  \\
       \d W \\
     \end{array}
   \right) \label{eq:Ddz} \\
\nonumber   \mathbb{D}(\d W)&=& -\frac{\sqrt{-1}B}{2A}VY^{-1}(\d Z,\d W^t)\left(
     \begin{array}{cc}
       Y^{-1}V^tVY^{-1} & -Y^{-1}V^t \\
      -VY^{-1} & I \\
     \end{array}
   \right)\left(
     \begin{array}{c}
       \d Z  \\
       \d W  \\
     \end{array}
   \right) \\
   &&\qquad\qquad-  \sqrt{-1}dWY^{-1}\d Z \label{eq:DdW}
\end{eqnarray}
\end{theorem}
We apply this theorem to invariant sections for Jacobi forms in section \ref{section:differential}. For details of this theorem, see section \ref{section:connection}.And the proof of this theorem is contained in section \ref{section:computation}.

Besides these, by considering the translation formula of the operators we obtained. We find a series of invariant differential operators for the Siegel-Jacobi forms of general degree. For Siegel modular space, Maass obtained the invariant differential operators in \cite{Maass}. Similar to his result, we have

\begin{theorem}[See Theorem \ref{theorem:invariant}]\label{theorem:introinvariant}
The operator matrix $Y_-Y_+$ is an invariant differential operator matrix on $\mathbb{H}_{n,m}$, thus each of the $(k,l)$ entries
of this matrix is an invariant differential operator on $\mathbb{H}_{n,m}$. And the operators $H^j,T_{k,l}^j,U_{k,l},V_{k,l}$ are all invariant differential
operators.
\end{theorem}

The notations in this Theorem are explained in section \ref{section:invariant}. We applied Maass's method to prove this theorem.

Another main result of this paper is the construction of two Maass-Shimura type differential operators.
In the Siegel case,  Maass and Shimura developed a differential operator which transforms a weight $k$ Siegel modular form to a weight $k+2$ non-holomorphic Siegel modular form. This operator is very useful in Shimura's theory of nearly holomorphic Siegel modular forms. As for Siegel-Jacobi form case, we construct two such operators in section \ref{section:shimura}.
\begin{theorem}[See Theorem \ref{theorem:h}, \ref{theorem:L}]
The operators
\begin{eqnarray*}
  H_{k,M}:&=& \det(Y)^{\kappa-k-1}\exp\left\{4\pi \Tr(MVY^{-1}V^t)\right\}\det\left(X_+\right)\left(\det(Y)^{k+1-\kappa}\exp\left\{-4\pi \Tr(MVY^{-1}V^t)\right\}\right), \\
  L_{k,M}:&=& \det(Y)^{\kappa'-k-1}\exp\left\{-4\pi \Tr(MVY^{-1}V^t)\right\}\det\left(-8\pi\sqrt{-1}\frac{\partial }{\partial Z}+\frac{\partial}{\partial W}M^{-1}(\frac{\partial }{\partial W})^t\right)\\
  & &(\det(Y)^{k+1-\kappa'}\exp\left\{-4\pi \Tr(MVY^{-1}V^t)\right\})
\end{eqnarray*}
 both map from $J_{k,M}$ to $J_{k,M+2}$. Here $\kappa=\frac{n+1}{2},\kappa'=\frac{n+m+1}{2}$.
\end{theorem}
In degree 1 case, the two operators become the operator $X_+$ and the heat operator respectively. They may be used to research the nearly holomorphic theory for Jacobi forms.

Differential operators play an important role in the theory of automorphic forms. They are closely related to L-functions. There are many results on differential operators for various of automorphic forms. For the case of  Siegel modular forms, holomorphic differential operators for Siegel modular forms and other automorphic forms are well studied by T.Ibukiyama in \cite{Ibukiyama}. In \cite{Bocherer}, S.B\"ocherer and S.Nagaoka used mod p differential operators for Siegel modular forms to study mod p properties. See also \cite{Bocherer2,Ibukiyama2} for more work in these areas.

As for Jacobi forms on $\mathbb{H}\times\mathbb{C}^m$, linear differential operators are given by Conley and Raum in \cite{Conley}. In \cite{Yang 2}, some invariant differential operators on the Siegel-Jacobi space are defined by J.-H Yang. Also Y.Choie and W.Eholzer studied the Rankin-Cohen brackets for Jacobi forms in \cite{Choie E}. In degree 1 case, S. B\"ocherer obtained the Rankin-Cohen brackets by Maass-Shimura type operators in \cite{Bocherer3}. We hope that similar results can be obtained by our Maass-Shimura type operator in general degree case. The first author thanks Professor T.Ibukiyama very much for pointing out this.

This paper is organized as follows.

In section \ref{section:connection}, we recall the concept of  connection and the metric in the Siegel-Jacobi space case, which will be used in our proof. The Chern connection $\mathbb{D}$ on $\mathbb{H}_{n,m}$ is given in Theorem \ref{theorem:connection}. The explicit proof is quite complicated and will be given in section \ref{section:computation}.
In section \ref{section:differential}, we consider the higher degree cases. Instead of the classical Hagson's method, we will use connections computed above to obtain differentials of Jacobi forms. First of all we get some operators in the determinant form, including both raising and lowering  operators, generalizing exactly the classical case. See Theorem \ref{theorem:R}.
Then using these results and Maass's work on Siegel modular forms, we obtain a series of invariant differential operators for Jacobi forms in section \ref{section:invariant}. These operators are summarized in  Theorem \ref{theorem:invariant}.
 In the last part of Section \ref{section:differential}, we define two analogues of the Maass-Shimura operator in the Siegel-Jacobi case, denoted by $H_{k,M}$ and $L_{k,M}$, which may be used in nearly holomorphic Jacobi forms theories.
See Theorem \ref{theorem:h} and \ref{theorem:L} and their proof.
Section \ref{section:computation} contains the explicit proof of Theorem \ref{theorem:introconnection}.

\section{Levi-Civita Connections of Siegel-Jacobi Space}\label{section:connection}
We recall some basic facts of  Riemannian  Geometry and K\"{a}hler geometry from \cite{Chern} and \cite{Kobayashi}. Suppose that $M$ is a smooth manifold, E is a vector bundle on $M$ and $\Gamma(E)$  is the set of global sections. A connection on $E$ is a map
\[\mathbb{D}:\Gamma(E)\rightarrow \Gamma(T(M^*)\otimes E)\]
satisfying the following two conditions:

(1) For any sections $s_1,s_2\in \Gamma(E)$,
\[\mathbb{D}(s_1+s_2)=\mathbb{D}(s_1)+\mathbb{D}(s_2).\]

(2) For any section $s\in\Gamma(E)$ and any $\alpha\in \mathcal{C}^\infty(M),$
\[\mathbb{D}(\alpha s)=d\alpha\otimes s+\alpha \mathbb{D}(s).\]

If M is a Riemannian manifold with a Riemannian metric $g=\sum\limits_{i,j}g_{ij}du^idu^j$, then there exists a unique torsion-free and metric-compatible connection, called the Levi-Civita connection whose Christoffel coefficients $\Gamma_{ij}^k$ satisfies
\begin{equation}\label{equation:oldgamma}
\Gamma_{ij}^k=\sum\frac{1}{2}g^{kl}\left(\frac{\partial g_{il}}{\partial u^j}+\frac{\partial g_{jl}}{\partial u^i}-\frac{\partial g_{ij}}{\partial u^l}\right),
\end{equation}
where $g^{ij}$ means the $i,j$ factor of the inverse matric of $(g_{ij})$. The connection matrix is defined to be
 $\omega=\left(\omega_i^j\right)$, where $ \omega_i^j=\sum\limits_k \Gamma_{ik}^jdu^k$.

Now suppose that $M$ is a Hermitian manifold. A Hermitian manifold means a complex manifold with a Hermitian metric on its holomorphic tangent space. Suppose $z_1,...z_n$ is a complex local coordinate system in $M$. Then the Hermitian metric is given by the form
\[h=\sum\limits_{\alpha,\beta}h_{\alpha\bar{\beta}}dz_{\alpha}dz_{\bar{\beta}}.\]
where $h_{\alpha\bar{\beta}}$ are the components of a positive-definite Hermitian matrix. The fundamental 2-form is given by
\[\Phi=\frac{i}{2}\sum\limits_{\alpha,\beta}h_{\alpha\bar{\beta}}dz_{\alpha}dz_{\bar{\beta}}\]
The manifold $M$ is K\"{a}hler if and only if the 2-form $\Phi$ is closed.

$M$ is also equipped with a unique torsion free connection, which is called the Chern connection. Its Christoffel coefficients satisfies
\[\Gamma^{\alpha}_{\beta\gamma}=\Gamma^{\alpha}_{\gamma\beta},\Gamma^{\overline{\alpha}}_{\overline{\beta}\overline{\gamma}}
=\Gamma^{\overline{\alpha}}_{\overline{\gamma}\overline{\beta}}\]
And others are all 0.

The Christoffel coefficients of Chern connections can be computed as
\begin{equation}\label{equation:gamma}
\Gamma^{\alpha}_{\gamma\beta}=\sum\limits_{i}h^{\alpha\bar{i}}\frac{\partial h_{\beta\bar{i}}}{\partial z_{\gamma}}
\end{equation}
where $(h^{\alpha\bar{i}})$ means the inverse matrix of the metric matrix of h.

The Hermitian metric naturally defines a Riemannian metric. If the manifold is K\"{a}hler, the Chern connection can induce a connection on the corresponding Riemannian manifold, and that is just the Levi-Civita connection.

Actually, the connection is divided into two parts: the holomorphic part $\mathbb{D}^{1,0}$ and the non-holomorphic part $\mathbb{D}^{0,1}$.  In general, the holomorphic and non-holomorphic part are conjugate to each other, and for simplicity, we will only consider the holomorphic part, and still write it as $\mathbb{D}$.

 The Chern connection of the Siegel-Jacobi space is given in Theorem \ref{theorem:connection}, whose proof is given in section \ref{section:computation}.

We first need to know the invariant metric of the Siegel-Jacobi space, in \cite{Yang}, J.-H Yang proved the following theorem:
\begin{theorem}\label{theorem:metric}
For any two positive real numbers $A$ and $B$, the following metric
\begin{displaymath}
\begin{aligned}
ds _{n,m;A,B}^2= & A\cdot \Tr(Y^{-1}dZY^{-1}d\overline{Z})\\
&+B\left\{\Tr(Y^{-1}V^tVY^{-1}dZY^{-1}d\overline{Z})+\Tr(Y^{-1}(dW)^td\overline{W})\right\}\\
&-B\left\{\Tr(VY^{-1}dZY^{-1}(d\overline{W}^t))+\Tr(VY^{-1}d\overline{Z}Y^{-1}(dW)^t)\right\}
\end{aligned}
\end{displaymath}
is a Riemannian metric on $\mathbb{H}_{n,m}$, which is invariant under the action of the Jacobi group $G^J$.
\end{theorem}
The symbols $Y,V$ denotes the imaginary parts of $Z$ and $W$ respectively, and
\[dZ=\left(
     \begin{array}{cccc}
       dz_{1,1} & dz_{1,2} &\ldots & dz_{1,n} \\
       dz_{2,1} & \ddots & & \vdots\\
       \vdots & & \ddots & \\
       dz_{n,1} &\ldots & & dz_{n,n} \\
     \end{array}
   \right),
   \
   dW=\left(
     \begin{array}{cccc}
       dw_{1,1} & dw_{1,2} &\ldots & dw_{1,n} \\
       dw_{2,1} & \ddots & & \vdots\\
       \vdots & & \ddots & \\
       dw_{m,1} &\ldots & & dw_{m,n} \\
     \end{array}
   \right).
\]

In fact, this metric is  K\"{a}hler and we can compute its Chern connection.
The Chern connection associated to this metric $ds _{n,m;A,B}^2$, denoted by $\mathbb{D}$, is invariant under the action of $G^J$, i.e. $g\mathbb{D}=\mathbb{D}g$ for any $g\in G^J$. Especially, $\mathbb{D}$ maps invariant sections to invariant sections, and we can use this fact to get differential operators for Jacobi forms as the Siegel modular case given in \cite{Yang Yin}.

We first look at the case when $n=m=1$. In this case, the invariant metric becomes
\[ds^2_{A,B}=Ay^{-2}dzd\bar{z}+B\left(v^2y^{-3}dzd\bar{z}+y^{-1}dwd\bar{w}-vy^{-2}dzd\bar{w}-vy^{-2}dwd\bar{z}\right).\]

So the metric matrix is written as
\[h=\left(
     \begin{array}{cc}
       \frac{A}{y^2}+\frac{Bv^2}{y^3} &  -\frac{Bv}{y^2} \\
       -\frac{Bv}{y^2} & \frac{B}{y} \\
          \end{array}
   \right),\]

and

\[h^{-1}=\left(
     \begin{array}{cccc}
        \frac{y^2}{A}  & \frac{vy}{A} \\
       \frac{vy}{A}  & \frac{y}{B}+\frac{v^2}{A} \\
     \end{array}
   \right).\]

By the formula (\ref{equation:gamma}) , we can get the connection matrix $w$ with $w_{i}^j=\sum\limits_k\Gamma_{ik}^jdu^k$

\[
w=\frac{B\sqrt{-1}}{2A}
\left( \begin{smallmatrix}
(\frac{2A}{By}+\frac{v^2}{y^2})dz-\frac{v}{y}dw &  \frac{v^3}{y^3}dz+(-\frac{v^2}{y^2}+\frac{A}{By})dw  \\
             -\frac{v}{y}dz+dw & (-\frac{v^2}{y^2}+\frac{A}{By})dz+\frac{v}{y}dw \\
      \end{smallmatrix} \right).
\]

Thus we get
\begin{displaymath}
\begin{aligned}
 \mathbb{D}(dz)&=-\left(\frac{\sqrt{-1}}{y}+\frac{\sqrt{-1}Bv^2}{2Ay^2}\right)dz^2+\frac{\sqrt{-1}Bv}{Ay}dzdw-\frac{\sqrt{-1}B}{2A}dw^2\\
&=(dz,dw)\left(
     \begin{array}{cc}
       -\frac{\sqrt{-1}}{y}-\frac{\sqrt{-1}Bv^2}{2Ay^2} & \frac{\sqrt{-1}Bv}{2Ay} \\
       \frac{\sqrt{-1}Bv}{2Ay} & -\frac{\sqrt{-1}B}{2A} \\
     \end{array}
   \right)\left(
     \begin{array}{c}
       dz  \\
       dw  \\
     \end{array}
   \right),\\
  \mathbb{D}(dw)&=-\frac{\sqrt{-1}Bv^3}{2Ay^3}dz^2-\left(\frac{-\sqrt{-1}Bv^2}{Ay^2}+\frac{\sqrt{-1}}{y}\right)dzdw-\frac{\sqrt{-1}Bv}{2Ay}dw^2\\
&=(dz,dw)\left(
     \begin{array}{cc}
       -\frac{\sqrt{-1}Bv^3}{2Ay^3} & -\frac{-\sqrt{-1}Bv^2}{2Ay^2}-\frac{\sqrt{-1}}{2y} \\
      -\frac{-\sqrt{-1}Bv^2}{2Ay^2}-\frac{\sqrt{-1}}{2y} & -\frac{\sqrt{-1}Bv}{2Ay} \\
     \end{array}
   \right)\left(
     \begin{array}{c}
       dz  \\
       dw  \\
     \end{array}
   \right)\\
&=(dz,dw)\left(
     \begin{array}{cc}
       -\frac{\sqrt{-1}Bv^3}{2Ay^3} & -\frac{-\sqrt{-1}Bv^2}{2Ay^2} \\
      -\frac{-\sqrt{-1}Bv^2}{2Ay^2} & -\frac{\sqrt{-1}Bv}{2Ay} \\
     \end{array}
   \right)\left(
     \begin{array}{c}
       dz  \\
       dw  \\
     \end{array}
   \right)-\frac{\sqrt{-1}}{y}dzdw.\\
   \end{aligned}
\end{displaymath}
For higher degree cases, we have the following results:
\begin{theorem}\label{theorem:connection}
Let $\mathbb{D}$ be the Chern connection on the manifold $\mathbb{H}_{n,m}$ associated to the invariant metric in Theorem \ref{theorem:metric}, then $\mathbb{D}$ satisfies
\begin{eqnarray}
   \mathbb{D}(\d Z)&=&-\frac{\sqrt{-1}B}{2A}(\d Z,\d W^t)\left(
     \begin{array}{cc}
      2\frac{A}{B}Y^{-1}+Y^{-1}V^tVY^{-1} & -Y^{-1}V^t \\
       -VY^{-1} & I\\
     \end{array}
   \right)\left(
     \begin{array}{c}
       \d Z  \\
       \d W \\
     \end{array}
   \right) \label{eq:Ddz} \\
\nonumber   \mathbb{D}(\d W)&=& -\frac{\sqrt{-1}B}{2A}VY^{-1}(\d Z,\d W^t)\left(
     \begin{array}{cc}
       Y^{-1}V^tVY^{-1} & -Y^{-1}V^t \\
      -VY^{-1} & I \\
     \end{array}
   \right)\left(
     \begin{array}{c}
       \d Z  \\
       \d W  \\
     \end{array}
   \right) \\
   &&\qquad\qquad- \sqrt{-1}\d WY^{-1}\d Z \label{eq:DdW}
\end{eqnarray}

   where $\mathbb{D}(\d Z)$ means $\left(
     \begin{array}{cccc}
       \mathbb{D}(dz_{11}) & \mathbb{D}(dz_{12}) &\ldots &\mathbb{D}(dz_{1n}) \\
      \mathbb{D}(dz_{21}) & \ddots & &\vdots \\
     \vdots & & \ddots &\\
     \mathbb{D}(dz_{n1})&\ldots& &\mathbb{D}(dz_{nn})\\
     \end{array}
   \right)$ and $\mathbb{D}(dW)$ is similar.
\end{theorem}
   The proof of the theorem by direct computation is given in the last section. We will use this theorem to compute
differentials of Jacobi forms and get differential operators.

\section{Differential Operators on Siegel-Jacobi Space}\label{section:differential}
\subsection{Action of Connection on Invariant Jacobi Forms}\label{section:differential1}
First, we will consider how the action of the Chern connection we obtained acts on an invariant function on $\mathbb{H}_{n,m}$.
Let $h\in J_{0,0}$ be an invariant Jacobi form. Then its image under the connection map is also invariant under the action of $\Gamma^J$, which is
\begin{equation}
\mathbb{D}(h)=\Tr\left(\frac{\partial h}{\partial Z}dZ\right)+\Tr\left(\frac{\partial h}{\partial W}dW\right).
\end{equation}

Here we use the notations
$$\frac{\partial}{\partial W}:=\left(
     \begin{array}{cccc}
       \frac{\partial}{\partial w_{1,1}} & \frac{\partial}{\partial w_{2,1}} &\ldots & \frac{\partial}{\partial w_{m,1}} \\
       \frac{\partial}{\partial w_{1,2}} & \ddots & & \vdots\\
       \vdots & & \ddots & \\
       \frac{\partial}{\partial w_{1,n}} &\ldots & & \frac{\partial}{\partial w_{m,n}} \\
     \end{array}
   \right),$$
   $$\frac{\partial}{\partial Z}:=\left(
     \begin{array}{cccc}
       \frac{\partial}{\partial z_{1,1}} & \frac{\partial}{2\partial z_{1,2}} &\ldots & \frac{\partial}{2\partial z_{1,n}} \\
       \frac{\partial}{2\partial z_{1,2}} & \ddots & & \vdots\\
       \vdots & & \ddots & \\
       \frac{\partial}{2\partial z_{1,n}} &\ldots & & \frac{\partial}{\partial z_{n,n}} \\
     \end{array}
   \right).$$

 Now we consider the translation formula of $dZ$ and $dW$. Recall the action of an element $g=\left(\left(
     \begin{array}{cc}
       A & B \\
       C & D \\
     \end{array}
   \right),(\lambda,\mu,\kappa)\right)\in \Gamma^J,$ we have its action of $\mathbb{H}_{m,n}$ is given by the formula (\ref{equation:action}). Then for an element $(Z,W)\in \mathbb{H}_{m,n}$, denote its image under the action of $g$ by $(\tilde{Z},\widetilde{W})$. It is  checked in the section 2 of \cite{Yang} that
 \begin{eqnarray}\label{equation:dwdz}
 \nonumber d(\tilde{Z})&=& ((CZ+D)^{-1})^{t}dZ(CZ+D)^{-1} \\
 d(\widetilde{W}) &=&  dW(CZ+D)^{-1}+\{\lambda-(W+\lambda Z+\mu)(CZ+D)^{-1}C\}dZ(CZ+D)^{-1}.
 \end{eqnarray}
 Since the $\d W$ term does not come from $\d (\tilde{Z})$, considering the $\d W$ term under the action of $\Gamma^J$, we can see the translation relation as follows:
   \begin{equation}\label{equation: partialW}
   \frac{\partial h((\widetilde{Z},\widetilde{W}))}{\partial \widetilde{W}}\d W=(CZ+D)\frac{\partial h(Z,W)}{\partial W}\d W.
   \end{equation}

If $m=n$, we can take determinant of the both sides, and then we get
\[\det\left(\widetilde{\frac{\partial h}{\partial W}}\right)=\det(CZ+D)\det\left(\frac{\partial h}{\partial W}\right).\]
This means $\det\left(\frac{\partial h}{\partial W}\right)\in J_{1,0}$.

If $m>n$, we can not take the determinant directly. But we can choose $n$ different rows of $\frac{\partial h}{\partial W}$, and take determinant of the new square matrixes so we can get $C_m^n$ operators in this form.

Next, we consider the twice differential $\mathbb{D}^2(h)$, which is
\begin{equation}
\mathbb{D}^2(h)=\Tr\left(\frac{\partial h}{\partial Z}\mathbb{D}(\d Z)+\mathbb{D}\left(\frac{\partial h}{\partial Z}\right)\d Z\right)+
\Tr\left(\frac{\partial h}{\partial W}\mathbb{D}(\d W)+\mathbb{D}\left(\frac{\partial h}{\partial W}\right)\d W\right).
\end{equation}
 Here $\mathbb{D}\left(\frac{\partial h}{\partial Z}\right)$ means the connection acts on every entry of the matrix.

 This is invariant under the action of $\Gamma^J$. By the discussion for the translation formula of $dZ$ and $dW$ above, we can easily see that the terms consisting of $\d w_i\d w_j$ are also invariant. Applying the expression for $\mathbb{D}(\d Z)$ and $\mathbb{D}(\d W)$ in Theorem \ref{theorem:connection}, we see that
 \begin{displaymath}
 \begin{aligned}
 \mathbb{D}^2(h)=&-\Tr\left(\frac{\sqrt{-1}B}{2A}\frac{\partial h}{\partial Z}(\d Z,\d W^t)\left(
     \begin{array}{cc}
      2\frac{A}{B}Y^{-1}+Y^{-1}V^tVY^{-1} & -Y^{-1}V^t \\
       -VY^{-1} & I\\
     \end{array}
   \right)\left(
     \begin{array}{c}
       \d Z  \\
       \d W \\
     \end{array}
   \right) +\mathbb{D}\left(\frac{\partial h}{\partial Z}\right)\d Z\right)\\
   &-\Tr\left(\frac{\sqrt{-1}B}{2A}\frac{\partial h}{\partial W}VY^{-1}(\d Z,\d W^t)\left(
     \begin{array}{cc}
       Y^{-1}V^tVY^{-1} & -Y^{-1}V^t \\
      -VY^{-1} & I \\
     \end{array}
   \right)\left(
     \begin{array}{c}
       \d Z  \\
       \d W  \\
     \end{array}
   \right)
   +\sqrt{-1}\d WY^{-1}\d Z\right)\\
   &-\Tr\left(\mathbb{D}\left(\frac{\partial h}{\partial W}\right)\d W\right)
 \end{aligned}
 \end{displaymath}

 Although the expression is complicated, we can see the $\d w_i\d w_j$ terms only show up in three parts in the above formula, which are:

 \[\Tr\left(\frac{\sqrt{-1}B}{2A}\frac{\partial h}{\partial Z} \d W^t\d W\right)+\Tr\left(\frac{\sqrt{-1}B}{2A}\frac{\partial h}{\partial W}VY^{-1} \d W^t\d W\right)+\Tr\left(\mathbb{D}\left(\frac{\partial h}{\partial W}\right)\d W\right)\]
 By the definition of $\Tr$, we can express the above section as
\begin{equation}\label{equation:transform}
\Tr\left(\frac{\sqrt{-1}B}{2A}\frac{\partial h}{\partial Z} \d W^t\d W\right)+\Tr\left(\frac{\sqrt{-1}B}{2A}\frac{\partial h}{\partial W}VY^{-1} \d W^t\d W\right)+\sum_{1\leq i\leq m}\sum_{1\leq j\leq n}\d\left(\frac{\partial h}{\partial w_{ij}}\right)\d w_{ij}.
\end{equation}

The last term can be revised again. Denote the $i$-row vector of $\d W$ by $\d W_i$. We have $\d\left(\widetilde{W}_i\right)=\d W_i(CZ+D)^{-1}+T$, where $T$ is the $\d Z$ term. Thus by calculation the last term $\sum_{1\leq i\leq m}\sum_{1\leq j\leq n}\d\left(\frac{\partial h}{\partial w_{ij}}\right)\d w_{ij}$  above can be rewritten as
\[\Tr\left(\sum_{i\leq j}\frac{\partial}{\partial W_i}\left(\frac{\partial h}{\partial W_j}\right)^t\d W_j^t \d W_i\right),\]
with $\frac{\partial}{\partial W_i}(\frac{\partial h}{\partial W_j})^t:=\left(
     \begin{array}{cccc}
       \frac{\partial^2 h}{\partial w_{j1}\partial w_{i1}} & \frac{\partial^2 h}{\partial w_{j2}\partial w_{i1}} &\ldots &\frac{\partial^2 h}{\partial w_{jn}\partial w_{i1}} \\
      \frac{\partial^2 h}{\partial w_{j1}\partial w_{i2}} & \ddots & &\vdots \\
     \vdots & & \ddots &\\
     \frac{\partial^2 h}{\partial w_{j1}\partial w_{in}}&\ldots& &\frac{\partial^2 h}{\partial w_{jn}\partial w_{in}}\\
     \end{array}
   \right)$ for any function $h$.

So now the formula (\ref{equation:transform}) becomes
\begin{equation}\label{realtransform}
   \Tr\left(\frac{\sqrt{-1}B}{2A}\frac{\partial h}{\partial Z} \d W^t\d W\right)+\Tr\left(\frac{\sqrt{-1}B}{2A}\frac{\partial h}{\partial W}VY^{-1} \d W^t\d W\right)+\Tr\left(\sum_{i\leq j}\frac{\partial}{\partial W_i}\left(\frac{\partial h}{\partial W_j}\right)^t\d W_j^t\d W_i\right).
\end{equation}

Moveover, from translation formula (\ref{equation: partialW}), and the translation of $\d W$, we can easily deduce that the last term is invariant itself already, ignoring the extra $\d Z$ part. Thus we obtain the section
\begin{equation}\label{equation:d2}
\Tr\left(\frac{\partial h}{\partial Z}\d W^t\d W\right)+\Tr\left(\frac{\partial h}{\partial W}VY^{-1}\d W^t\d W\right).
\end{equation}
We will get derivatives of Jacobi forms from this  section.

\subsection{Action of Connections on General Jacobi forms}\label{section:differential2}
Let $f\in J_{k,M}^{hol}$ be a holomorphic Jacobi form. It is known that $h:=fh_1\bar{f}$ is an invariant form on $\mathbb{H}_{n,m}$ (see \cite{Ziegler}, page 202), where
\[h_1=\det(Y)^k\exp\left\{-4\pi \cdot\Tr(MVY^{-1}V^t)\right\}.\]

We will compute the explicit expression of $\det\left(\frac{\partial h}{\partial W}\right)$ and of formula (\ref{equation:d2}) for this $h$. First we will compute $\det\left(\frac{\partial h}{\partial W}\right)$. We have the following lemma
\begin{lemma}
$\frac{\partial \left\{\Tr\left(MVY^{-1}V^t\right)\right\}}{\partial W}=-\sqrt{-1}Y^{-1}V^tM$
\end{lemma}
\begin{proof}

Denote $Y^{-1}$ by $R$. As
\[\Tr\left(MVY^{-1}V^t\right)=\sum_{a,b,r,s}M_{ab}V_{br}R_{rs}V_{as},\]
So
\begin{displaymath}
\begin{aligned}
 \frac{\partial \left\{\Tr(MVY^{-1}V^t)\right\}}{\partial W_{ij}}=&\frac{\partial\left\{ \sum_{a,b,r,s}M_{ab}V_{br}R_{rs}V_{as}\right\}}{\partial W_{ij}}\\
 =&-\frac{\sqrt{-1}}{2}\sum_{b,r}M_{ib}V_{br}R_{rj}-\frac{\sqrt{-1}}{2}\sum_{a,s}M_{ai}R_{js}V_{as}\\
 =&-\sqrt{-1}\sum_{b,r}M_{ib}V_{br}R_{rj}.
 \end{aligned}
\end{displaymath}

Since $M,Y$ are symmetric, we have $\frac{\partial\left\{ \Tr\left(MVY^{-1}V^t\right)\right\}}{\partial W}=-\sqrt{-1}Y^{-1}V^tM.$
\end{proof}
Furthermore, as $h_1=\det(Y)^k\exp(-4\pi \Tr\left(MVY^{-1}V^t\right)),$
we have
\[\frac{\partial h_1}{\partial W}=4\sqrt{-1}\pi Y^{-1}V^tMh_1,\]
and thus
\begin{displaymath}
\begin{aligned}
\det\left(\frac{\partial h}{\partial W}\right)=& \det\left(\frac{\partial f}{\partial W}h_1\overline{f}+\frac{\partial h_1}{\partial W}f\overline{f}+\frac{\partial \overline{f}}{\partial W}h_1f\right)\\
=&\det\left(\frac{\partial f}{\partial W}h_1\overline{f}+4\pi \sqrt{-1}(Y^{-1}V^tM)h_1f\overline{f}\right)\in J_{1,0}.
\end{aligned}
\end{displaymath}
As $h$ is invariant, we see that $\det\left(\frac{\partial f}{\partial W}/f+4\pi \sqrt{-1}Y^{-1}V^tM\right)$ is in $J_{1,0}$. So
\begin{equation}
R_1(f):=\det\left(\frac{\partial f}{\partial W}+4\pi \sqrt{-1}(Y^{-1}V^tM)f\right)\in J_{nk+1,nM}.
\end{equation}

The method above works for holomorphic Jacobi forms only, but the results also hold for non-holomorphic Jacobi forms.  Now we consider a not necessarily holomorphic Jacobi form $f$. As we discussed above $\frac{\partial (fh_1\overline{f})}{\partial W}$ and $\frac{\partial f}{\partial W}h_1\overline{f}+4\pi \sqrt{-1}(Y^{-1}V^tM)h_1f\overline{f}$ both satisfy the formula (\ref{equation: partialW}), and so is their difference $\frac{\partial \bar{f}}{\partial W}h_1f$.Then we know that
 \[\det\left(\frac{\partial \bar{f}}{\partial W}h_1f\right)\in J_{1,0}.\]
  Since $\det(Y)^{-2}\det(dZ)\det(d\bar{Z})$ is invariant under the Jacobi group (see \cite{Maass}), dividing it in the conjugation of the above Jacobi form, we see that
\[\left(\det\left(\frac{\partial f}{\partial \overline{W}}h_1\bar{f}\right)\right)^2\det\left(Y^{2}\right)\det\left(dZ\right)^{-1}\]
is invariant under the action of $\Gamma^J$. Thus
\[L_1:=\det\left(\frac{\partial f}{\partial \overline{W}}Y\right)\]
is a Jacobi form in $J_{nk-1,nM}$.

The two operators $R_1$ and $L_1$ are generalization of $Y_+$ and $Y_-$ in the introduction respectively.

Next we consider the explicit expression of formula (\ref{equation:d2}). Now this becomes

\begin{equation}\label{equation:d2final}
\Tr\left(\frac{\partial f}{\partial Z}h_1\bar{f}dW^tdW\right)+\Tr\left(\frac{\partial h_1}{\partial Z}f\bar{f}dW^tdW\right)+\Tr\left(\frac{\partial f}{\partial W}h_1\bar{f}VY^{-1}dW^tdW\right)+\Tr\left(\frac{\partial h_1}{\partial W}f\bar{f}VY^{-1}dW^tdW\right).
\end{equation}

This equals
\begin{displaymath}
\begin{aligned}
&\Tr\left(\left(\frac{\partial f}{\partial Z}-\frac{\sqrt{-1}k}{2}fY^{-1}-2\pi\sqrt{-1}Y^{-1}V^tMVY^{-1}f+4\pi\sqrt{-1}Y^{-1}V^tMVY^{-1}f+\frac{\partial f}{\partial W}VY^{-1}\right)h_1\bar{f}dW^tdW\right)\\
=&\Tr\left(\left(\frac{\partial f}{\partial Z}-\frac{\sqrt{-1}k}{2}fY^{-1}+2\pi\sqrt{-1}Y^{-1}V^tMVY^{-1}f+\frac{1}{2}\frac{\partial f}{\partial W}VY^{-1}+\frac{1}{2}Y^{-1}V^t\frac{\partial f}{\partial W}^t\right)h_1\bar{f}dW^tdW\right),
\end{aligned}
\end{displaymath}
by the following lemma.

\begin{lemma}\label{lemma}
We have
\begin{itemize}
  \item[(a)]~~ $\frac{\partial R_{st}}{\partial Z_{kl}}=2^{-\delta(k,l)-1}\sqrt{-1}\left(R_{kt}R_{sl}+R_{ks}R_{tl}\right)$, where $R=Y^{-1}$, and $\delta(i,j)$ is the Dirac symbol.
  \item[(b)]~~ $\frac{\partial\left\{ \Tr\left(MVY^{-1}V^t\right)\right\}}{\partial Z}=\frac{\sqrt{-1}}{2}Y^{-1}V^tMVY^{-1}$.
  \item[(c)]~~ $\frac{\partial \left\{\det(Y)\right\}}{\partial Z}=-\frac{\sqrt{-1}}{2}\det(Y)Y^{-1}$.
\end{itemize}
\end{lemma}
\begin{proof}
For (a), $R$ is uniquely determined by the equations:
$$\sum_{s}Y_{is}R_{st}=\delta(i,t)$$
where $\delta$ is the Dirac symbol with
\begin{displaymath}
\delta(i,t)= \left\{ \begin{array}{ll}
0 & \textrm{if $i\neq t$}\\
1 & \textrm{if $i=t$}
\end{array} \right.
\end{displaymath}
Taking derivations on both sides, we get:
$$\sum_s \frac{\partial Y_{is}}{\partial Z_{kl}}R_{st}+\sum_s Y_{is}\frac{\partial R_{st}}{\partial Z_{kl}}=0$$
That is
\begin{displaymath}
\left\{ \begin{array}{lll}
\sum_s Y_{is}\frac{\partial R_{st}}{\partial Z_{kl}}=0     & \textrm {if $i \neq k,l$}\\
-\frac{\sqrt{-1}}{2}R_{lt}+\sum_s Y_{ks}\frac{\partial R_{st}}{\partial Z_{kl}}=0 & \textrm {if $i=k$}\\
-\frac{\sqrt{-1}}{2}R_{kt}+\sum_s Y_{ls}\frac{\partial R_{st}}{\partial Z_{kl}}=0 & \textrm {if $i=l $}
\end{array} \right.
\end{displaymath}

These equations determine $\frac{\partial R_{st}}{\partial Z_{kl}}$ uniquely. And it is easily checked that if we take $\frac{\partial R_{st}}{\partial Z_{kl}}=2^{-\delta(k,l)-1}\sqrt{-1}(R_{kt}R_{sl}+R_{ks}R_{tl})$, then the equations above are all satisfied, and this finishes the proof of (a).

For (b), by (a), we have:
\begin{displaymath}
\begin{aligned}
\left(\frac{\partial \Tr(MVY^{-1}V^t)}{\partial Z}\right)_{kl}
&=2^{\delta(k,l)-1}\frac{\partial \left(\sum_{a,b,t,s}M_{ab}V_{bs}R_{st}V_{at}\right)}{\partial Z_{kl}}\\&=\frac{\sqrt{-1}}{4}\sum_{a,b,s,t}M_{ab}V_{bs}(R_{kt}R_{sl}+R_{ks}R_{tl})V_{at}\\
&=\frac{\sqrt{-1}}{2}\sum_{a,b,s,t}R_{kt}V_{at}M_{ab}V_{bs}R_{sl}
\end{aligned}
\end{displaymath}
and the statement follows.

As for (c), it is easily checked that $\frac{\partial (\det(Y))}{\partial Z_{i,j}}=-2^{\delta_{ij}-2}\sqrt{-1}(Y_{i,j}^\ast+Y_{j,i}^\ast)$, with $Y_{i,j}^\ast$ the cofactor of $Y_{i,j}$. So $\frac{\partial (\det(Y))}{\partial Z}=-\frac{\sqrt{-1}}{2}\det(Y)Y^{-1}$ follows from the definition of $\frac{\partial}{\partial Z}$.
\end{proof}

 For any symmetric matrix function $A$, if we have $\Tr(A\d Z)$ invariant under the action of $\Gamma^J$, then we have
 \[\Tr(\widetilde{A}\d(\tilde{Z}))=\Tr\left(\widetilde{A}((CZ+D)^{-1})^{t}\d Z(CZ+D)^{-1}\right)=\Tr\left(
 (CZ+D)^{-1}\widetilde{A}((CZ+D)^{-1})^{t}\d Z\right)=\Tr(A\d Z)\]

 This means that
 \[\Tr\left((\widetilde{A}-(CZ+D)A(CZ+D)^t)\d Z\right)=0,\]
 thus $(\widetilde{A}-(CZ+D)A(CZ+D)^t)$ has to be 0 and $A$ satisfies the translation law of
 \[\widetilde{A}=(CZ+D)A(CZ+D)^t\]
 with $\widetilde{A}=g(A), g=\left(\left(
     \begin{array}{cc}
       A & B \\
       C & D \\
     \end{array}
   \right),(\lambda,\mu,\kappa)\right)\in \Gamma^J$.

We also know that
\[\widetilde{(\d W^t\d W)}=((CZ+D)^{-1})^{t}(\d W^t\d W)(CZ+D)^{-1}+T\]
where $T$ is some differential forms in coordinates of $dZ$. So the above argument works for $\d W^t\d W$ as well. More precisely, if we have $\Tr(A\d W^t\d W)$ is invariant ignoring its $\d Z$ part in the translation formula, then by the same argument as above, $A$ also have to satisfy the translation formula
\[\widetilde{A}=(CZ+D)A(CZ+D)^t\]

Apply this to (\ref{equation:d2final}), we see that

\[\left(\frac{\partial f}{\partial Z}-\frac{\sqrt{-1}k}{2}fY^{-1}+2\pi\sqrt{-1}Y^{-1}V^tMVY^{-1}f+\frac{1}{2}\frac{\partial f}{\partial W}VY^{-1}+\frac{1}{2}Y^{-1}V^t\frac{\partial f}{\partial W}^t\right)\]
satisfies this translation formula.

Then taking determinant of the above matrix, we obtain that

\[R_2(f):=\det\left(\frac{\partial f}{\partial Z}-\frac{\sqrt{-1}k}{2}fY^{-1}+2\pi\sqrt{-1}Y^{-1}V^tMVY^{-1}f+\frac{1}{2}\frac{\partial f}{\partial W}VY^{-1}+\frac{1}{2}Y^{-1}V^t\frac{\partial f}{\partial W}^t\right)\]
is an element in $J_{nk+2,nM}$

\begin{remark}
Here we use the symmetric matrix of $\frac{\partial h}{\partial Z}+\frac{\partial h}{\partial W}VY^{-1}$.
But actually one can show that $\frac{\partial h}{\partial Z}+\frac{\partial h}{\partial W}VY^{-1}$ works already. We choose the symmetric matrix because it seems more natural.
\end{remark}

As we pointed out before, what we get are also right for non-holomorphic Jacobi forms. From this we see that
$\Tr\left(\frac{\partial \bar{f}}{\partial Z}dW^tdW\right)fh_1+\Tr\left(\frac{\partial \bar{f}}{\partial W}VY^{-1}dW^tdW\right)fh_1$ is invariant ignoring the $dZ$ part. So the determinant
 \[\det\left(\frac{\partial \bar{f}}{\partial Z}fh_1+\frac{1}{2}\frac{\partial \bar{f}}{\partial W}VY^{-1}fh_1+\frac{1}{2}Y^{-1}V^t\frac{\partial \overline{f}}{\partial W}^tfh_1\right)\in J_{2,0}\]
Taking conjugation and dividing $\det(Y)^{-2}\det(dZ)\det(d\bar{Z})$ as we did before, we get
\[L_2(f):=\det\left(\frac{\partial f}{\partial \bar{Z}}Y^2+\frac{1}{2}\frac{\partial f}{\partial \overline{W}}VY+\frac{1}{2}YV^t\frac{\partial f}{\partial \overline{W}}^t\right)\in J_{nk-2,nM}\]

Summary the results above, we have the following theorem
\begin{theorem}\label{theorem:R}

\begin{itemize}
    \item[(a)] If $n=m$, for any $f\in J_{k,M}$, $R_1(f)=\det\left(\frac{\partial f}{\partial W}+4\pi \sqrt{-1}(Y^{-1}V^tM)f\right)\in J_{nk+1,nM}$;
  \item[(b)]If $n=m$, with the same $f$, $L_1(f)=\det\left(\frac{\partial f}{\partial \overline{W}}Y\right)$ is in $J_{nk-1,nM}$.
   \item[(c)]$R_2(f)=\det\left(\frac{\partial f}{\partial Z}-\frac{\sqrt{-1}k}{2}fY^{-1}+2\pi\sqrt{-1}Y^{-1}V^tMVY^{-1}f+\frac{1}{2}\frac{\partial f}{\partial W}VY^{-1}+\frac{1}{2}Y^{-1}V^t\frac{\partial f}{\partial W}^t\right)\in J_{nk+2,nM}$
   \item[(d)]$L_2(f)=\det\left(\frac{\partial f}{\partial \bar{Z}}Y^2+\frac{1}{2}\frac{\partial f}{\partial \overline{W}}VY+\frac{1}{2}YV^t\frac{\partial f}{\partial \overline{W}}^t\right)\in J_{nk-2,nM}$
\end{itemize}
\end{theorem}
If $n$ and $m$ equals 1, then these four operators $R_1,L_1,R_2,L_2$ in this theorem  are just the operators $Y_+,Y_-,X_+,X_-$ we introduced in the introduction. Though they are not linear, we can use these to construct invariant differential operators and Maass-Shimura operators.

\subsection{Invariant Differential Operators}\label{section:invariant}
In this subsection we will study the invariant differential operators for the Siegel-Jacobi space, that is, differential
operators invariant under the action of the Jacobi group. For the space $\mathbb{H}_{1,m}$, things are clear, they just come from the composition of raising and lowering operators, as is shown in proposition 2.8 of \cite{Conley}.
While for higher degree cases, not much is known. In the Siegel modular space case, Maass has constructed invariant differential operators successfully. And we can apply his method to the Siegel-Jacobi space case. Together with what we discussed above for the higher degree Jacobi forms, we are able to get invariant differential operators generalizing Maass's results.

Set
\begin{eqnarray*}
Y_+&=&\frac{\partial}{\partial W},\quad Y_{+,k}=\frac{\partial}{\partial W_k},\quad Y_-=\frac{\partial}{\partial\overline{W}}^tY,\quad Y_{-,k}=\frac{\partial}{\partial\overline{W}_k}^tY,\\
  X_+&=& 2\sqrt{-1}\frac{\partial}{\partial Z}+\sqrt{-1}Y^{-1}V^t\frac{\partial}{\partial W}+\left(\sqrt{-1}Y^{-1}V^t\frac{\partial}{\partial W}\right)^t, \\
  X_-&=&Y(Y\overline{Y}_+)^t,\\
  K &=& 2\sqrt{-1}Y\frac{\partial}{\partial Z}+\sqrt{-1}V^t\frac{\partial}{\partial W}^t+\sqrt{-1}\left(Y^{-1}V^t\frac{\partial}{\partial W}^tY\right)^t=YX_+,\\
  \Lambda &=& \sqrt{-1}Y\frac{\partial}{\partial \overline{Z}}+\sqrt{-1}V^t\frac{\partial}{\partial \overline{W}}^t+\sqrt{-1}\left(Y^{-1}V^t\frac{\partial}{\partial \overline{W}}^tY\right)^t.
\end{eqnarray*}

Then from sections \ref{section:differential1} and \ref{section:differential2}, we have
\begin{prop}
\begin{eqnarray*}
  \widetilde{Y}_+&=& (CZ+D)Y_+;\quad\widetilde{Y}_-=Y_-(CZ+D)^{-1}; \\
  \widetilde{Y}_{+,k} &=& (CZ+D)Y_{+,k};\quad\widetilde{Y}_{-,k}=Y_{-,k}(CZ+D)^{-1}; \\
  \widetilde{X}_+ &=& (CZ+D)\left((CZ+D)X_+\right)^t; \quad\widetilde{X}_-=\left((CZ+D)^{-1}\right)^{t}\left((CZ+D)^{-1}X_-\right)^t; \\
  \widetilde{K}&=& \left((C\overline{Z}+D)^{-1}\right)^t\left((CZ+D)K^t\right)^t;\quad \widetilde{\Lambda}=\left((CZ+D)^{-1}\right)^t\left((C\overline{Z}+D)\Lambda^t\right)^t.
\end{eqnarray*}

Here $\widetilde{}$ means after the action of $g=\left(\left(
     \begin{array}{cc}
       A & B \\
       C & D \\
     \end{array}
   \right),(\lambda,\mu,\kappa)\right)\in \Gamma^J$.
\end{prop}
\begin{proof}
The statement for $Y_+$ is obvious from formula(\ref{equation: partialW}), so
\[\frac{\partial}{\partial \widetilde{\overline{W}}}=(C\bar{Z}+D)\frac{\partial}{\partial \overline{W}}.\]
And we have
\begin{equation}
\widetilde{Y}=\left((C\bar{Z}+D)^{-1}\right)^tY(CZ+D)^{-1}=\left((CZ+D)^{-1}\right)^tY(C\bar{Z}+D)^{-1}
\end{equation}
Thus $\widetilde{Y}_-=Y_-(CZ+D)^{-1}.$ The statements for $Y_{+,k}$ and $Y_{-,k}$ can also be checked easily.
In formula(\ref{equation:d2}), we know that for any $h$ which is invariant under the action of $\Gamma^J$,
$\Tr\left(X_+(h)dZ\right)$ is also invariant. This means that
\[\widetilde{X}_+=(CZ+D)\left((CZ+D)X_+\right)^t,\]
thus
\begin{eqnarray*}
  \widetilde{K} &=& \left((C\overline{Z}+D)^{-1}\right)^t\left((CZ+D)K^t\right)^t; \\
  \widetilde{\Lambda} &=& \left((CZ+D)^{-1}\right)^t\left((C\overline{Z}+D)\Lambda^t\right)^t;\\
  \widetilde{X}_- &=&\left((CZ+D)^{-1}\right)^t\left((CZ+D)^{-1}X_-\right)^t
\end{eqnarray*}

\end{proof}

$\Tr(\Lambda K)$ is not invariant, but we have relation
\[\left((C\overline{Z}+D)\Lambda^t\right)^t=\Lambda\left(C\overline{Z}+D)^t-\frac{n+1}{2}(Z-\overline{Z}\right)C^t,\]
so we get
\[\widetilde{\Lambda}\widetilde{K}=\left((CZ+D)^{-1}\right)^t\left(\Lambda\left(C\overline{Z}+D\right)^t-\frac{n+1}{2}(CZ+D)^t+\frac{n+1}{2}(C\overline{Z}+D)^t\right)\left((C\overline{Z}+D)^{-1}\right)^t\left((CZ+D)K^t\right)^t\]
which can be written as
\[\widetilde{\Lambda}\widetilde{K}+\frac{n+1}{2}\widetilde{K}=\left((CZ+D)^{-1}\right)^t\left((CZ+D)\left(\Lambda K+\frac{n+1}{2}K\right)^t\right)^t.\]
Thus
\[\Tr\left(\widetilde{\Lambda}\widetilde{K}+\frac{n+1}{2}\widetilde{K}\right)=\Tr\left(\Lambda K+\frac{n+1}{2}K\right).\]
Following H.Maass, we set
\[A^{(1)}=\left(\Lambda K+\frac{n+1}{2}K\right).\]
\[A^{(j)}=A^{(1)}A^{(j-1)}-\frac{n+1}{2}\Lambda A^{(j-1)}+\frac{1}{2}\Lambda \Tr\left(A^{(j-1)}\right)+\frac{1}{2}(Z-\overline{Z})\left(
(Z-\bar{Z})^{-1}(\Lambda^tA^{(j-1)t})^t\right)^t.\]
By the same argument as in page 111 to 116 of \cite{Maass}, one can show that
\[\widetilde{A}^{(j)}=\left((CZ+D)^{-1}\right)^t\left((CZ+D)A^{(j)t}\right)^t,\]
and \[H^j:=\Tr(A^{(j)})\] is invariant under the action of the Jacobi group.

Moreover, if we set
\[T_{k,l}^j=\Tr(Y_{-,k}^tY_{+,l}^tA^{(j)}),\quad U_{k,l}=\Tr(Y_{-,k}^tY_{-,l}X_+),\quad V_{k,l}=\Tr(Y_{+,k}Y_{+,l}^tX_-),\]
then they are all invariant differential operators.
\begin{theorem}\label{theorem:invariant}
The operator matrix $Y_-Y_+$ is an invariant differential operator matrix on $\mathbb{H}_{n,m}$, thus each of the $(k,l)$ entries
of this matrix is an invariant differential operator on $\mathbb{H}_{n,m}$. And the operators $H^j,T_{k,l}^j,U_{k,l},V_{k,l}$ are all invariant differential
operators.
\end{theorem}
\begin{remark}
The invariant operator $Y_-Y_+$ has already been known in proposition 4.2 of \cite{Yang 2}. We hope that all the invariant differential operators come as the combination of the operators in Theorem \ref{theorem:invariant} as in the case of $\mathbb{H}_{1,m}$.
\end{remark}
\subsection{Maass-Shimura Type Differential Operators for Jacobi Forms}\label{section:shimura}

The Maass-Shimura differential operators are operators defined for Siegel modular forms. Shimura used this operator to study the properties of nearly holomorphic Siegel modular forms and obtained the algebraicity of values of Siegel modular forms. Let $\mathbb{H}_n$ be the usual Siegel upper half plane with coordinate $Z=(z_{ij})$. The imaginary part of $Z$ is denoted by $Y$ as before. Let $f$ be a Siegel modular form of weight $k$ on $\mathbb{H}_n$,
the Maass-Shimura differential operator acts on $f$ as
\[\det(Y)^{\kappa-k-1}\det\left(\frac{\partial}{\partial Z}\right)(\det(Y)^{k+1-\kappa}f)\]
where $\kappa$ equals $\frac{n+1}{2}$. The Maass-Shimura operator maps a weight $k$ modular form to a modular form of weight $k+2$.

Now we can first define a similar differential operator for Jacobi forms by using the results above. Let
\[H_{k,M}:=\det(Y)^{\kappa-k-1}\exp\left\{4\pi \Tr(MVY^{-1}V^t)\right\}\det\left(X_+\right)\left(\det(Y)^{k+1-\kappa}\exp\left\{-4\pi \Tr(MVY^{-1}V^t)\right\}\right)\]

Then we have the following theorem
\begin{theorem}\label{theorem:h}
$H_{k,M}$ is a differential operator mapping from $J_{k,M}$ to $J_{k+2,M}$.
\end{theorem}
\begin{proof}
We show that the theorem is true by comparing it with the Maass-Shimura operator.

First, consider the case $k=0, M=0$. Now the operator becomes
\[H_{0,0}=\det(Y)^{\kappa-1}\det\left(X_+\right)(\det(Y)^{1-\kappa})\]
Recall that $X_+=2\sqrt{-1}\frac{\partial}{\partial Z}+\sqrt{-1}Y^{-1}V^t\frac{\partial}{\partial W}+\Big(\sqrt{-1}Y^{-1}V^t\frac{\partial}{\partial W}\Big)^t,$ and under the action of $G^J$, we have
\[\widetilde{X}_+=(CZ+D)\left((CZ+D)X_+\right)^t\]

For Siegel modular forms, the classical Maass-Shimura operator is
\[\det(Y)^{\kappa-1}\det\left(\frac{\partial}{\partial Z}\right)(\det(Y)^{1-\kappa}).\]
And $\frac{\partial}{\partial Z}$ satisfies the translation law
$\frac{\partial}{\partial \widetilde{Z}}=(CZ+D)\left((CZ+D)\frac{\partial}{\partial Z}\right)^t,$ which is formally the
same as $X_+$. Since $X_+$ acts the same as $\frac{\partial}{\partial Z}$ on $CZ+D$, so $\det(X_+)$ and $\det\left(\frac{\partial}{\partial Z}\right)$ also satisfy the same translation law. Moreover, $\widetilde{Y}=\left((C\bar{Z}+D)^{-1}\right)^tY(CZ+D)^{-1}$, so its composition with $\frac{\partial}{\partial W}$ does not change the translation formula.
So $H_{0,0}$ and the Maass-Shimura operator satisfies the same translation law formally. This means that $H_{0,0}$ maps from $J_{0,0}$ to $J_{2,0}$, which is the basic case of our theorem.

For general weight and index case, recall that we have the invariant form
\[h=f\det(Y)^k\exp\left\{-4\pi \Tr(MVY^{-1}V^t)\right\}\bar{f}.\]

So apply $H_{0,0}$ to $h$, we get that
\[\det(Y)^{\kappa-1}\det\left(X_+\right)\left(\det(Y)^{1-\kappa}f\det(Y)^k\exp\left\{-4\pi \Tr(MVY^{-1}V^t)\right\}\bar{f}\right)\]
is an element of $J_{2,0}$. Then multiplying $\frac{f}{h}$, this becomes
\[H_{k,M}f=\det(Y)^{\kappa-k-1}\exp\left\{4\pi \Tr(MVY^{-1}V^t)\right\}\det\left(X_+\right)\left(\det(Y)^{k+1-\kappa}\exp\left\{-4\pi \Tr(MVY^{-1}V^t)\right\}f)\right),\]
which is now a Jacobi form of weight $k+2$ and index $M$. Thus we have proved the theorem.

\end{proof}

This operator can be viewed as a generalization of the Maass-Shimura operator if we restrict a Jacobi form from $\mathbb{H}_{n,m}$ to $\mathbb{H}_{n}$. Also this is a generalization of the operator $X_+$ in the introduction.

By viewing the  group $\Gamma^J$ as a subgroup of $\Sp(g,\mathbb{Z})$, we can define another analogue of the Maass-shimura operator. In the degree 1 case, this is just the heat operator. It is  defined as follows
\begin{eqnarray*}
L_{k,M}:&=&\det(Y)^{\kappa'-k-1}\exp\left\{-4\pi \Tr(MVY^{-1}V^t)\right\}\det\left(-8\pi\sqrt{-1}\frac{\partial }{\partial Z}+\frac{\partial}{\partial W}M^{-1}(\frac{\partial }{\partial W})^t\right)\\
& &(\det(Y)^{k+1-\kappa'}\exp\left\{-4\pi \Tr(MVY^{-1}V^t)\right\})
\end{eqnarray*}
where $\kappa'=\frac{n+m+1}{2}$.
\begin{theorem}\label{theorem:L}
$L_{k,M}$ is a differential operator mapping from $J_{k,M}$ to $J_{k+2,M}$.
\end{theorem}
\begin{proof}
Following from (\cite{Maass},P.317), Maass showed that for Siegel upper half plane of degree $g$, we have the following transformation formula,  by the action of an $S=\left(
     \begin{array}{cc}
       A & B \\
       C & D \\
     \end{array}
   \right)\in \Sp(g,\mathbb{Z})$,
\begin{equation}\label{equation:determinant}
\left|\frac{\partial }{\partial \widetilde{Z}} \right|=\left|CZ+D\right|^{\frac{g+3}{2}}\left|\frac{\partial }{\partial Z}\right||CZ+D|^{\frac{1-g}{2}},
\end{equation} where $\widetilde{Z}$ is the image of $Z$ by the action of $S$.

      Now embed $\mathbb{H}_{n,m}$ into $\mathbb{H}_{n+m}$ by sending $(Z,W)$ to $\left(
     \begin{array}{cc}
       Z & W \\
       W^t & Z' \\
     \end{array}
   \right)$.
   Also, embed $G^J=\Sp(n,\mathbb{R})\ltimes H_\mathbb{R}^{(n,m)}$ into $\Sp(n+m,\mathbb{R})$ in the classical way by sending $(M,(\lambda,\mu;\kappa))$  to

   \[\left(
     \begin{array}{cccc}
       A & 0 & B & \mu^t \\
       \lambda & I & \mu & \kappa \\
       C & 0 & D & -\lambda^t \\
       0 & 0 & 0 & I \\
     \end{array}
   \right),\]
   where $M=\left(
     \begin{array}{cc}
       A & B \\
       C & D \\
     \end{array}
   \right).$

   We know the action of $G^J$ on $\mathbb{H}_{n,m}$ coincides with the action of $\Sp(n+m,\mathbb{R})$ on $\mathbb{H}_{n+m}$.  Now fix a symmetric matrix $M$. Applying formula(\ref{equation:determinant}), we see that $\left|
     \begin{array}{cc}
       \frac{\partial}{\partial \widetilde{Z}} & \frac{\partial}{\partial \widetilde{W}} \\
       \frac{\partial}{\partial \widetilde{W}}^t & M \\
     \end{array}
   \right|=|CZ+D|^{\frac{g+3}{2}}\left|
     \begin{array}{cc}
       \frac{\partial}{\partial Z} & \frac{\partial}{\partial W} \\
       \frac{\partial}{\partial W}^t & M \\
     \end{array}
   \right||CZ+D|^{\frac{1-g}{2}}.$ Combining this with $|\widetilde{Y}|=|Y||CZ+D|^{-1}|C\overline{Z}+D|^{-1}$, we see that
   the operator

   \[l_{k,M}:=|Y|^{\kappa'-1}\left|
     \begin{array}{cc}
       \frac{\partial}{\partial \widetilde{Z}} & \frac{\partial}{\partial \widetilde{W}} \\
       \frac{\partial}{\partial \widetilde{W}}^t & M \\
     \end{array}
   \right||Y|^{1-\kappa'}\]
   maps a $\Gamma^J$-invariant form to a weight 2 Jacobi form. Thus $l_{k,M}\left(f\det(Y)^k\exp\left\{-4\pi \Tr(MVY^{-1}V^t)\right\}\bar{f}\right)$ is a weight 2 Jacobi form and
   so is $l_{k,M}\left(f\det(Y)^k\exp\left\{-4\pi \Tr(MVY^{-1}V^t)\right\}\bar{f}\right)/\left(f\det(Y)^k\exp\left\{-4\pi \Tr(MVY^{-1}V^t)\right\}\bar{f}\right)$, we then deduce Theorem (\ref{theorem:L}) by the following lemma.
    \end{proof}
\begin{lemma}
 $\left|
     \begin{array}{cc}
       \frac{\partial}{\partial Z} & \frac{\partial}{\partial W} \\
       \frac{\partial}{\partial W}^t & M \\
     \end{array}
   \right|$ can be expressed as $\left|\frac{\partial}{\partial Z}+\frac{\partial}{\partial W}M^{-1}\frac{\partial}{\partial W}^t\right|\cdot|M|$
\end{lemma}
\begin{proof}
Multiplying $\left(
     \begin{array}{cc}
       M & 0 \\
       -M^{-1}\frac{\partial}{\partial W}^tM & M^{-1} \\
     \end{array}
   \right)$ to $\left(
     \begin{array}{cc}
       \frac{\partial}{\partial Z} & \frac{\partial}{\partial W} \\
       \frac{\partial}{\partial W}^t & M \\
     \end{array}
   \right)$. We get the result we need easily.
\end{proof}

\section{Computation of the connections}\label{section:computation}
In this section we give the proof of Theorem \ref{theorem:connection} by direct computation. The notations and ideas are similar to \cite{Yang Yin}. Set $\Omega=\{(i,j)| 1\leq i\leq j\leq n\}, \Omega'=\{(i',j')|1\leq i'\leq m,1\leq j'\leq n\}$, and fix the notation
$I=(i,j)$, $J=(r,s)$, $K=(p,q)$, $L=(a,b)\in\Omega$; $I'=(i',j')$, $J=(r',s')$, $K'=(p',q')$, $L'=(a',b')\in\Omega'$.
In the following, we define $Z_I:=Z_{ij},W_{I'}=W_{i'j'}$.

Let $R:=(R_{ij})_{n\times n}=Y^{-1}$. Then the metric on the Siegel-Jacobi space in Theorem \ref{theorem:metric} is given by
\begin{displaymath}
\begin{aligned}
ds_{n,m;A,B}^2=&A\Tr(RdZRd\overline{Z})\\
&+B\left(\Tr(RV^tVRdZRd\overline{Z})+\Tr(R(dW)^td\overline{W})
-\Tr(VRdZR(d\overline{W}^t))-\Tr(VRd\overline{Z}R(dW)^t)\right)\\
=&A\sum_{i=1}^n\sum_{j=1}^n\sum_{r=1}^n\sum_{s=1}^n R_{ir}R_{js}dZ_{ij}d\bar{Z}_{rs}\\
&+B\sum_{a=1}^n\sum_{b=1}^m\sum_{k=1}^n\sum_{i=1}^n\sum_{j=1}^n\sum_{r=1}^n\sum_{s=1}^n R_{sa}V_{ba}V_{bk}R_{ki}dZ_{ij}R_{jr}d\bar{Z}_{rs}\\
&+B\sum_{r=1}^n\sum_{j=1}^n\sum_{i=1}^mR_{rj}dW_{ij}d\bar{W}_{ir}-B\sum_{r=1}^m\sum_{k=1}^n\sum_{i=1}^n\sum_{j=1}^n\sum_{s=1}^nV_{rk}R_{ki}dZ_{ij}R_{js}d\bar{W}_{rs}\\
&-B\sum_{r=1}^m\sum_{k=1}^n\sum_{i=1}^n\sum_{j=1}^n\sum_{s=1}^nV_{rk}R_{ki}\bar{dZ_{ij}}R_{js}dW_{rs}\\
=&A\sum_{i\leq j}\sum_{r\leq s}2^{1-\delta(i,j)-\delta(r,s)}\times (R_{ir}R_{js}+R_{jr}R_{is})d\bar{Z}_{rs}dZ_{ij}\\
&+B\sum_{i\leq j}\sum_{r\leq s}\sum_{k=1}\sum_{l=1}\sum_{a=1} 2^{-\delta(i,j)-\delta(r,s)}(R_{sk}V_{ak}V_{al}R_{li}R_{jr}\\
&+R_{sk}V_{ak}V_{al}R_{lj}R_{ir}+R_{rk}V_{ak}V_{al}R_{li}R_{js}+R_{rk}V_{ak}V_{al}R_{lj}R_{is})dZ_{ij}d\bar{Z}_{rs}\\
&+B\sum_{r=1}^n\sum_{j=1}^n\sum_{i=1}^mR_{rj}dW_{ij}d\bar{W}_{ir}\\
&-B\sum_{r=1}\sum_{k=1}\sum_{i\leq j}\sum_{s=1}2^{-\delta(i,j)}V_{rk}(R_{ki}R_{js}+R_{kj}R_{is})dZ_{ij}d\bar{W}_{rs}\\
&-B\sum_{r=1}\sum_{k=1}\sum_{i\leq j}\sum_{s=1}2^{-\delta(i,j)}V_{rk}(R_{ki}R_{js}+R_{kj}R_{is})d\bar{Z}_{ij}dW_{rs}
\end{aligned}
\end{displaymath}

We have the following proposition.
\begin{prop}
The metric above is K\"{a}hler.
\end{prop}

\begin{proof}
This proposition follows from the above expression of the metric and lemma (\ref{lemma}). See also \cite{Berceanu} for more geometric properties of the manifold.

To prove the metric is K\"{a}hler, we have to prove that the closed form $\omega$ associated to the metric is closed.

We first show that the $dZ_{ij}\wedge dZ_{pq}\wedge d\bar{Z}_{rs}$ part of $d\omega$ is 0. If we denote the coefficients of $dZ_{ij}d\bar{Z}_{rs}$ in $ds_{n,m;A,B}^2$ above by $\phi(i,j,r,s)$, then this equals to say
\begin{eqnarray}\label{eq:identity}
\frac{\partial\phi(i,j,r,s)}{\partial Z_{pq}}=\frac{\partial\phi(p,q,r,s)}{\partial Z_{ij}}\end{eqnarray}
Since $\phi(i,j,r,s)$ obviously has two parts: $2^{1-\delta(i,j)-\delta(r,s)} (R_{ir}R_{js}+R_{jr}R_{is})$ and 
\[\sum_{k=1}\sum_{l=1}\sum_{a=1} 2^{-\delta(i,j)-\delta(r,s)}(R_{sk}V_{ak}V_{al}R_{li}R_{jr}
+R_{sk}V_{ak}V_{al}R_{lj}R_{ir}+R_{rk}V_{ak}V_{al}R_{li}R_{js}+R_{rk}V_{ak}V_{al}R_{lj}R_{is}).\]
We first compute the partial derivative of $2^{1-\delta(i,j)-\delta(r,s)} (R_{ir}R_{js}+R_{jr}R_{is})$ with respect to $Z_{pq}$. By lemma (\ref{lemma}), this is
\begin{eqnarray}\nonumber
2^{-\delta(i,j)-\delta(r,s)-\delta(p,q)}\sqrt{-1}(&R_{ip}R_{qr}R_{js}+R_{iq}R_{pr}R_{js}+R_{ir}R_{jp}R_{qs}+R_{ir}R_{jq}R_{ps}\\
&+R_{jp}R_{rq}R_{is}+R_{jq}R_{rp}R_{is}+R_{jr}R_{ip}R_{qs}+R_{jr}R_{iq}R_{ps})\nonumber\end{eqnarray}
By similar calculation, we see that the expression for the partial derivative of $2^{1-\delta(p,q)-\delta(r,s)} (R_{pr}R_{qs}+R_{qr}R_{ps})$ with respect to $Z_{ij}$ is also the formula above. So the first part for (\ref{eq:identity}) holds.

For the other part, we will compute the partial derivative of
\[\sum_{k=1}\sum_{l=1}\sum_{a=1} 2^{-\delta(i,j)-\delta(r,s)}(R_{sk}V_{ak}V_{al}R_{li}R_{jr}
+R_{sk}V_{ak}V_{al}R_{lj}R_{ir}+R_{rk}V_{ak}V_{al}R_{li}R_{js}+R_{rk}V_{ak}V_{al}R_{lj}R_{is})\]
Using Lemma(\ref{lemma}), we see that this is $\sum_{k=1}\sum_{l=1}\sum_{a=1} 2^{-\delta(i,j)-\delta(r,s)-\delta(p,q)-1}\sqrt{-1}V_{ak}V_{al}$
\begin{eqnarray*}
&(R_{sp}R_{aq}R_{ki}R_{jr}+
R_{sq}R_{ap}R_{ki}R_{jr}+R_{sa}R_{kp}R_{qi}R_{jr}+R_{sa}R_{kq}R_{pi}R_{jr}+R_{sa}R_{jp}R_{ki}R_{qr}+R_{sa}R_{jq}R_{ki}R_{pr}\\
&+R_{sp}R_{aq}R_{kj}R_{ir}+R_{sq}R_{ap}R_{kj}R_{ir}+R_{sa}R_{kp}R_{qj}R_{ir}+R_{sa}R_{kq}R_{pj}R_{ir}+R_{sa}R_{ip}R_{kj}R_{qr}+R_{sa}R_{iq}R_{kj}R_{pr}\\
&+R_{rp}R_{aq}R_{ki}R_{js}+
R_{rq}R_{ap}R_{ki}R_{js}+R_{ra}R_{kp}R_{qi}R_{js}+R_{ra}R_{kq}R_{pi}R_{js}+R_{ra}R_{jp}R_{ki}R_{qs}+R_{ra}R_{jq}R_{ki}R_{ps}\\
&+R_{rp}R_{aq}R_{kj}R_{is}+
R_{rq}R_{ap}R_{kj}R_{is}+R_{ra}R_{kp}R_{qj}R_{is}+R_{ra}R_{kq}R_{pj}R_{is}+R_{ra}R_{ip}R_{kj}R_{qs}+R_{ra}R_{iq}R_{kj}R_{ps})
\end{eqnarray*}
Similarly, this also equals the partial derivative of the second part of $\phi(p,q,r,s)$ with respect to $Z_{ij}$.
Thus we see that formula (\ref{eq:identity}) holds and the $dZ_{ij}\wedge dZ_{pq}\wedge d\bar{Z}_{rs}$ part is 0.

Next we prove that the $dZ_{ij}\wedge dW_{pq}\wedge d\bar{W}_{rs}$ part is 0. One of these coefficients comes from the partial derivative of $dW_{pq}\wedge d\bar{W}_{rs}$, and this equals $\delta(p,r)2^{-\delta(i,j)-1}\sqrt{-1}(R_{si}R_{qj}+R_{sj}R_{qi})$.

Others come from the partial derivative of $dZ_{ij}\wedge d\bar{W}_{rs}$ with respect to $W_{pq}$. By computation, this part  equals $-\delta(p,r)2^{-\delta(i,j)-1}\sqrt{-1}(R_{si}R_{qj}+R_{sj}R_{qi})$.
So the $dZ_{ij}\wedge dW_{pq}\wedge d\bar{W}_{rs}$ part is also 0.

Next we consider the $dZ_{ij}\wedge dZ_{pq}\wedge d\bar{W}_{rs}$ part. This equals 0 means that the partial derivative of
$\sum\limits_{k=1}2^{-\delta(i,j)}V_{rk}(R_{ki}R_{js}+R_{kj}R_{is})$ with respect to $Z_{pq}$ equals the partial derivative of
$\sum\limits_{k=1}2^{-\delta(p,q)}V_{rk}(R_{kp}R_{qs}+R_{kq}R_{ps})$ with respect to $Z_{ij}$. By computation, they both equal
\begin{eqnarray*}
\sum\limits_{k=1}2^{-\delta(i,j)-\delta(p,q)-1}\sqrt{-1}V_{rk}(&R_{kp}R_{qi}R_{js}+R_{kq}R_{ip}R_{js}+R_{ki}R_{jp}R_{qs}
+R_{ki}R_{jq}R_{sp}\\
+&R_{kp}R_{qj}R_{is}+R_{kq}R_{jp}R_{is}+R_{kj}R_{ip}R_{qs}
+R_{kj}R_{iq}R_{sp})
\end{eqnarray*}
So this part is also 0.

Next is the $dZ_{ij}\wedge dW_{pq}\wedge d\bar{Z}_{rs}$ part. They come from the partial derivative of the $dZ_{ij}\wedge d\bar{Z}_{rs}$ part and the $ dW_{pq}\wedge d\bar{Z}_{rs}$ part. By the same way of computation, they both equal
\begin{eqnarray*}
2^{-\delta(r,s)-\delta(i,j)-1}\sqrt{-1}\sum\limits_{k}V_{pk}(&R_{ki}R_{jr}R_{sq}+R_{kj}R_{ir}R_{sq}+R_{kr}R_{is}R_{jq}+R_{kr}R_{js}R_{iq}\\
&+R_{ki}R_{js}R_{rq}+R_{kj}R_{is}R_{rq}+R_{ks}R_{ir}R_{jq}+R_{ks}R_{jr}R_{iq})
\end{eqnarray*}

The $dW_{ij}\wedge dW_{pq}\wedge d\bar{W}_{rs}$ part is obviously 0, so we only need to prove the $dW_{ij}\wedge dW_{pq}\wedge d\bar{Z}_{rs}$ part is 0 now, other cases are just conjugations of the proved ones. This part comes from the partial derivative of $dW_{ij}\wedge d\bar{Z}_{rs}$ and $ dW_{pq}\wedge d\bar{Z}_{rs}$, and it is easy to check that this part is also 0.

Combining all these, we have seen that $d\omega=0$, and so the metric is  K\"{a}hler.
\end{proof}

Now the Hermitian-metric matrix associated to this metric is given by
\[W=\left(
     \begin{array}{cc}
        W^1  & W^2 \\
       ^tW^{2}  & W^3 \\
     \end{array}
   \right)\]
where $W^1=(W^1_{I,\bar{J}})_{I,J\in\Omega}$, and
\begin{displaymath}
\begin{aligned}
W^1_{I,\bar{J}}=&A2^{1-\delta(i,j)-\delta(r,s)}\times (R_{ir}R_{js}+R_{jr}R_{is})\\
&+B\sum_{k=1}^n\sum_{l=1}^n\sum_{a=1}^m 2^{-\delta(i,j)-\delta(r,s)}
V_{ak}V_{al}(R_{sk}R_{li}R_{jr}+R_{sk}R_{lj}R_{ir}+R_{rk}R_{li}R_{js}+R_{rk}R_{lj}R_{is}).
\end{aligned}
\end{displaymath}
$W^2=(W^2_{I,\bar{J'}})_{I\in\Omega,J'\in\Omega'}$ \[W^2_{I,\bar{J'}}=-B\sum_{k=1}^n2^{-\delta(i,j)}V_{r'k}(R_{ki}R_{js'}+R_{kj}R_{is'}),\]
$W^3=(W^3_{I',\bar{J'}})_{I',J'\in\Omega'}$, and
\[W^3_{I',\bar{J'}}=B\delta(i',s')R_{r'j'}\]

To compute the connection, we have to know the inverse of $W$. Let $M^1=(M^1_{I,\bar{J}})_{I,J\in\Omega}$, $M^2=(M^2_{I,\bar{J'}})_{I\in\Omega J'\in\Omega'}$ , and $M^3=(M^3_{I',\bar{J'}})_{I',J'\in\Omega'}$, with
\begin{eqnarray*}
  M^1_{I,\bar{J}} &=& \frac{1}{2A}Y_{ir}Y_{js}+\frac{1}{2A}Y_{jr}Y_{is} \\
  M^2_{I,\bar{J'}}&=& \frac{1}{2A}V_{r'i}Y_{js'}+\frac{1}{2A}V_{r'j}Y_{is'} \\
  M^3_{I',\bar{J'}} &=& \sum_{k=1}^n\sum_{l=1}^n\frac{1}{2A}V_{rk}R_{kl}V_{il}Y_{js}+\frac{1}{2A}V_{ri}V_{js}+\frac{1}{B}\delta(i,r)Y_{js}
\end{eqnarray*}

\begin{lemma}
Let $M=\left(
     \begin{array}{cc}
        M^1  & M^2 \\
        ^tM^2 &  M^3 \\
     \end{array}
   \right)$, then $M$ is the inverse matrix of $W$.
\end{lemma}
\begin{proof}
(1) We use the Dirac symbol $\delta(I,J):= \left\{ \begin{array}{ll}
1 & \textrm{if $I=J \in \Omega$}\\
0 & \textrm{if $I\neq J \in \Omega$}
\end{array} \right.$. Recall that $I=(i,j), J=(r,s),$ and $K=(p,q).$ We have
\begin{displaymath}
\begin{aligned}
\sum_{J\in\Omega}M_{I\bar{J}}^1W_{J\bar{K}}^1=&\sum_{r\leqq s}M_{(i,j)\overline{(r,s)}}^1 W_{(r,s)\overline{(p,q)}}^1\\
=&\delta(I,K)+\frac{B}{A}\sum_{r\leqq s}(Y_{ir}Y_{js}+Y_{jr}Y_{is})\biggl(\sum_{k=1}^n\sum_{l=1}^n\sum_{a=1}^m 2^{-1-\delta(p,q)-\delta(r,s)}\times\\
&(R_{qk}V_{ak}V_{al}R_{lr}R_{sp}+R_{qk}V_{ak}V_{al}R_{ls}R_{rp}+R_{pk}V_{ak}V_{al}R_{lr}R_{sq}+R_{pk}V_{ak}V_{al}R_{ls}R_{rq})\biggl)\\
=&\delta(I,K)+\frac{B}{A}\sum_{r=1}^n\sum_{k=1}^n\sum_{l=1}^n\sum_{a=1}^m2^{-\delta(p,q)}\left(Y_{ir}Y_{jr}(R_{qk}V_{ak}V_{al}R_{lr}R_{rp}+R_{pk}V_{ak}V_{al}R_{lr}R_{rq})\right)\\
&+\frac{B}{A}\sum_{r<s}(Y_{ir}Y_{js}+Y_{jr}Y_{is})\biggl(\sum_{k=1}^n\sum_{l=1}^n\sum_{a=1}^m 2^{-1-\delta(p,q)}\times\\
&(R_{qk}V_{ak}V_{al}R_{lr}R_{sp}+R_{qk}V_{ak}V_{al}R_{ls}R_{rp}+R_{pk}V_{ak}V_{al}R_{lr}R_{sq}+R_{pk}V_{ak}V_{al}R_{ls}R_{rq})\biggl)\\
=&\delta(I,K)+\frac{B}{A}\sum_{r=1}^n\sum_{s=1}^n\sum_{k=1}^n\sum_{l=1}^n\sum_{a=1}^m2^{-1-\delta(p,q)}(Y_{ir}Y_{js}+Y_{jr}Y_{is})\\
&\times(R_{qk}V_{ak}V_{al}R_{ls}R_{rp}+R_{pk}V_{ak}V_{al}R_{ls}R_{rq})\\
=&\delta(I,K)+\frac{B}{A}\sum_{k=1}^n\sum_{l=1}^n\sum_{a=1}^m2^{-1-\delta(p,q)}V_{ak}V_{al}\\
&\times\left(\delta(i,p)\delta(j,l)R_{qk}+\delta(i,q)\delta(j,l)R_{pk}+\delta(j,p)\delta(i,l)R_{qk}+\delta(j,q)\delta(i,l)R_{pk}\right)
\end{aligned}
\end{displaymath}
\begin{displaymath}
\begin{aligned}
\sum_{J\in\Omega'}M_{I\bar{J'}}^2W_{K\bar{J'}}^2=&\sum_{r'=1}^m\sum_{s'=1}^nM_{(i,j)\overline{(r',s')}}^2W_{(p,q)\overline{(r',s')}}^2\\
=&-\frac{B}{A}\sum_{r'=1}^m\sum_{s'=1}^n(V_{r'i}Y_{js'}+V_{r'j}Y_{is'})\left(\sum_{k=1}^n2^{-1-\delta(i,j)}V_{r'k}(R_{ki}R_{js'}+R_{kj}R_{is'})\right)\\
=&-\frac{B}{A}\sum_{r'=1}^m\sum_{k=1}^n2^{-1-\delta(i,j)}\biggl(V_{r'i}V_{r'k}(\delta(j,q)R_{pk}+\delta(j,p)R_{qk})\\
&+V_{r'j}V_{r'k}(\delta(i,p)R_{qk}+\delta(i,q)R_{pk})\biggl)
\end{aligned}
\end{displaymath}
Adding these equalities together, we can see that
 \begin{equation}
 \sum_{J\in\Omega}M_{I\bar{J}}^1W_{J\bar{K}}^1+\sum_{J'\in\Omega'}M_{I\bar{J'}}^2W_{K\bar{J'}}^2=\delta(I,K)\end{equation}
(2)\ We will now compute $\sum_{J\in\Omega}M_{I\bar{J}}^1W_{\bar{J}K'}^2+\sum_{J'\in\Omega'}M_{I\bar{J'}}^2W_{J'\bar{K'}}^3.$
First we have
\begin{displaymath}
\begin{aligned}
\sum_{J\in\Omega}M_{I\bar{J}}^1W_{J\bar{K'}}^2=&\sum_{r\leq s}M_{(i,j)\bar{(r,s)}}^1W_{(r,s)\bar{(p',q')}}^2\\
=&\frac{1}{A}\sum_{r\leq s}(Y_{ir}Y_{js}+Y_{jr}Y_{is})\left(-B\sum_{k=1}^n2^{-1-\delta(r,s)}V_{p'k}(R_{kr}R_{sq'}+R_{sk}R_{rq'})\right)\\
=&-\frac{B}{2A}\sum_{k=1}^nV_{p'k}\left(\delta(i,k)\delta(j,q')+\delta(j,k)\delta(i,q')\right)\\
=&-\frac{B}{2A}V_{p'i}\delta(j,q')+2^{-1}V_{p'j}\delta(i,q'),
\end{aligned}
\end{displaymath}
and
\begin{displaymath}
\begin{aligned}
\sum_{J'\in\Omega'}M_{I\bar{J'}}^2W_{J'\bar{K'}}^3=&\sum_{r=1}^m\sum_{s=1}^nM_{(i,j)\overline{(r',s')}}^2W_{(r',s')\overline{(p',q')}}^3\\
=&\frac{B}{2A}\sum_{r=1}^m\sum_{s=1}^n(V_{r'i}Y_{js'}+V_{r'j}Y_{is'})\delta(r',p')R_{s'q'}\\
=&\frac{B}{2A}\left(V_{p'i}\delta(j,q')+V_{p'j}\delta(i,q')\right).
\end{aligned}
\end{displaymath}
So we have
\begin{displaymath}
\sum_{J\in\Omega}M_{I\bar{J}}^1W_{J\bar{K'}}^2+\sum_{J'\in\Omega'}M_{I\bar{J'}}^2W_{J'\bar{K'}}^3=0.
\end{displaymath}
(3)\ We compute $\sum_{J\in\Omega}M_{K'\bar{J}}^2W_{J\bar{I}}^1+\sum_{J'\in\Omega'}M_{K'\bar{J'}}^3W_{J'\bar{I}}^2$, which is more complicated. First we have
\begin{displaymath}
\begin{aligned}
\sum_{J\in\Omega}M_{K'\bar{J}}^2W_{J\bar{I}}^1=&\sum_{r\leq s}M_{(p',q')\overline{(r,s)}}^2W_{(r,s)\overline{(i,j)}}^1\\
=&\frac{1}{A}\sum_{r\leq s}(V_{p'r}Y_{sq'}+V_{p's}Y_{rq'})A2^{-\delta(i,j)-\delta(r,s)}\times (R_{ir}R{js}+R_{jr}R_{is})\\
&+\sum_{r\leq s}(V_{p'r}Y_{sq'}+V_{p's}Y_{rq'})B\sum_{k=1}^n\sum_{l=1}^n\sum_{a=1}^m 2^{-1-\delta(i,j)-\delta(r,s)}\times\\
&(R_{sk}V_{ak}V_{al}R_{li}R_{jr}+R_{sk}V_{ak}V_{al}R_{lj}R_{ir}+R_{rk}V_{ak}V_{al}R_{li}R_{js}+R_{rk}V_{ak}V_{al}R_{lj}R_{is})\\
=&\sum_{r=1}^n\sum_{s=1}^n2^{-\delta(i,j)}(V_{p'r}Y_{sq'}+V_{p's}Y_{rq'})R_{ri}R_{sj}\\
&+\frac{B}{A}\sum_{r=1}^n\sum_{s=1}^n(V_{p'r}Y_{sq'}+V_{p's}Y_{rq'})\sum_{k=1}^n\sum_{l=1}^n\sum_{a=1}^m 2^{-1-\delta(i,j)}
(R_{sk}V_{ak}V_{al}R_{li}R_{jr}+R_{sk}V_{ak}V_{al}R_{lj}R_{ir})\\
=&2^{-\delta(i,j)}\left(\sum_{r=1}^nV_{p'r}R_{ri}\delta(q',j)+\sum_{s=1}^nV_{p's}R_{sj}\delta(q',i)\right)+\frac{B}{A}2^{-1-\delta(i,j)-\delta(r,s)}\sum_{k=1}^n\sum_{l=1}^n\sum_{a=1}^mV_{ak}V_{al}\\
&\left(\sum_{r=1}^n(V_{p'r}R_{rj}R_{il}\delta(q',k)+V_{p'r}R_{ri}R_{jl}\delta(q',k))+\sum_{s=1}^n(V_{p's}R_{sk}R_{il}\delta(q',j)
+V_{p's}R_{sk}R_{jl}\delta(q',i))\right),
\end{aligned}
\end{displaymath}
and similarly
\begin{displaymath}
\begin{aligned}
\sum_{J'\in\Omega'}M_{K'\bar{J'}}^3W_{J'\bar{I}}^2=&\sum_{r'=1}^m\sum_{s'=1}^nM_{(p',q')\overline{(r',s')}}^3W_{(r',s')\overline{(i,j)}}^2\\
=&\sum_{r'=1}^m\sum_{s'=1}^n\left(\frac{1}{A}\sum_{k=1}^n\sum_{l=1}^nV_{r'k}R_{kl}V_{p'l}Y_{q's'}+\frac{1}{A}V_{r'p'}V_{q's'}+\frac{2}{B}\delta(p',r')Y_{q's'}\right)\\
&\left(-B\sum_{k=1}^n2^{-1-\delta(i,j)}V_{r'k}(R_{ki}R_{js'}+R_{kj}R_{is'})\right)\\
=&-\frac{B}{A}2^{-1-\delta(i,j)}\sum_{r'=1}^m\sum_{k=1}^n\sum_{l=1}^n\sum_{t=1}^n\left(V_{r'k}R_{kl}V_{p'l}V_{r't}(R_{ti}\delta(q',j)+R_{tj}\delta(q',i))\right)\\
&-\frac{B}{A}2^{-1-\delta(i,j)}\sum_{r'=1}^m\sum_{s'=1}^n\sum_{t=1}^nV_{r'p'}V_{q's'}V_{r't}(R_{ti}R_{js'}+R_{tj}R_{is'})\\
&-2^{-\delta(i,j)}\sum_{t=1}^nV_{p't}\left(R_{ti}\delta(q',j)+R_{tj}\delta(q',i)\right).
\end{aligned}
\end{displaymath}
Now it is easily checked that
\begin{equation}\sum_{J\in\Omega}M_{K'\bar{J}}^2W_{J\bar{I}}^1+\sum_{J'\in\Omega'}M_{K'\bar{J'}}^3W_{J'\bar{I}}^2=0
\end{equation}
(4)At last we calculate $\sum_{J\in\Omega}M_{J\bar{I'}}^2W_{J\bar{K'}}^2+\sum_{J'\in\Omega'}M_{I'\bar{J'}}^3W_{J'\bar{K'}}^3.$ First
\begin{displaymath}
\begin{aligned}
\sum_{J\in\Omega}M_{J\bar{I'}}^2W_{J\bar{K'}}^2=&\sum_{r\leq s}M_{(r,s)\overline{(i',j')}}^2W_{(r,s)\overline{(p',q')}}^2\\
=&-\frac{B}{A}\sum_{r\leq s}(V_{i'r}Y_{sj'}+V_{i's}Y_{rj'})\sum_{k=1}^n2^{-1-\delta(r,s)}V_{p'k}(R_{kr}R_{sq'}+R_{ks}R_{rp'})\\
=&-\frac{B}{A}2^{-1-\delta(r,s)}\sum_{r=1}^n\sum_{s=1}^n\sum_{k=1}^nV_{p'k}(V_{i'r}Y_{sj'}+V_{i's}Y_{rj'})R_{kr}R_{sq'}\\
=&-\frac{B}{2A}\sum_{l=1}^n\sum_{k=1}^n\left(V_{p'k}V_{i'l}R_{lk}\delta(q',j')+V_{p'k}V_{i'l}R_{lq'}\delta(k,j')\right),
\end{aligned}
\end{displaymath}
and
\begin{displaymath}
\begin{aligned}
\sum_{J'\in\Omega'}M_{I'\bar{J'}}^3W_{J'\bar{K'}}^3=&\sum_{r', s,}M_{(r',s')\bar{(i',j')}}^3W_{(r',s')\bar{(p',q')}}^3\\
=&\frac{B}{2}\sum_{r'=1}^m\sum_{s'=1}^n\biggl(\frac{1}{A}\sum_{k=1}^n\sum_{l=1}^nV_{r'k}R_{kl}V_{i'l}Y_{j's'}\\&+\frac{1}{A}V_{r'i'}V_{j's'}
+\frac{2}{B}\delta(i',r')Y_{j's'}\biggl)\delta(i',s')R_{r'j'}\\
=&\frac{B}{2A}\left(\sum_{k=1}^n\sum_{l=1}^nV_{p'k}V_{i'l}R_{kl}\delta(q',j')+\sum_{s'=1}^nV_{p'j'}V_{i's'}R_{s'q'}\right)+\delta(q',j')\delta(i',p').
\end{aligned}
\end{displaymath}
So
\begin{equation}\sum_{J\in\Omega}M_{J\bar{I'}}^2W_{J\bar{K'}}^2+\sum_{J'\in\Omega'}M_{I'\bar{J'}}^3W_{J'\bar{K'}}^3=\delta(I',K')
\end{equation}

Thus we have shown that $WM=I$.

\end{proof}
We will only prove the expression for $D(dZ)$ in Theorem \ref{theorem:connection}, the other part is the same.
Similarly to section 3.2 in \cite{Yang Yin}, we can compute $\Gamma_{I,J}^K$ as
\[\Gamma_{I,J}^K=\Gamma_{J,I}^K=-\sum_{L\in\Omega}W_{I\bar{L}}^1\frac{\partial M_{K\bar{L}}^1}{\partial Z_{J}}-\sum_{L'\in\Omega'}W_{I\bar{L'}}^2\frac{\partial M_{K\bar{L'}}^2}{\partial Z_{J}}\]
where $\Gamma_{J,I}^K$ means the Christoffel symbol  corresponding to the indices $J,I,K$.

Before the calculation, we define
\begin{displaymath}
\sigma_{(p,a)(r,s)}= \left\{ \begin{array}{ll}
1 & \textrm{if $Z_{pa}=Z_{rs}$}\\
0 & \textrm{if $Z_{pa}\neq Z_{rs}$}
\end{array} \right.
\end{displaymath}
Then now we have
\begin{displaymath}
\begin{aligned}
\sum_{L\in\Omega}W_{I\bar{L}}^1\frac{\partial M_{K\bar{L}}^1}{\partial Z_{J}}=&-\frac{\sqrt{-1}}{2}\sum_{a\leq b}A2^{-\delta(i,j)-\delta(a,b)}\times (R_{ia}R_{jb}+R_{ja}R_{ib})\\
&\times\frac{1}{A}\left(\sigma_{(p,a)(r,s)}Y_{qb}+\sigma_{(q,b)(r,s)}Y_{pa}+\sigma_{(q,a)(r,s)}Y_{pb}+\sigma_{(p,b)(r,s)}Y_{qa}\right)\\
&-\frac{\sqrt{-1}}{2}\sum_{a\leq b}B\sum_{k=1}^m\sum_{l=1}^n\sum_{t=1}^n 2^{-1-\delta(i,j)-\delta(a,b)}\times V_{kt}V_{kl}\\
&(R_{bl}R_{ia}R_{jt}+R_{ja}R_{it}R_{bl}+R_{ib}R_{jt}R_{al}+R_{jb}R_{it}R_{al})\times\\
&\frac{1}{A}\left(\sigma_{(p,a)(r,s)}Y_{qb}+\sigma_{(q,b)(r,s)}Y_{pa}+\sigma_{(q,a)(r,s)}Y_{pb}+\sigma_{(p,b)(r,s)}Y_{qa}\right)\\
=&-\frac{\sqrt{-1}}{2}\sum_{a=1}^n\sum_{b=1}^n2^{-\delta(i,j)}R_{ia}R_{jb}\left(Y_{qb}\sigma_{(p,a)(r,s)}+Y_{pa}\sigma_{(q,b)(r,s)}+Y_{pb}\sigma_{(q,a)(r,s)}+Y_{qa}\sigma_{(p,b)(r,s)}\right)\\
&-\frac{\sqrt{-1}}{2}\sum_{a=1}^n\sum_{b=1}^n\frac{B}{A}2^{-1-\delta(i,j)}\sum_{k=1}^m\sum_{l=1}^n\sum_{t=1}^nV_{kt}V_{kl}\biggl((R_{ia}R_{jt}+R_{ja}R_{it})\delta(q,l)\sigma_{(p,a)(r,s)}\\
&+(R_{it}R_{bl}+R_{jt}R_{bl})\delta(i,q)\sigma_{(p,b)(r,s)}\\
&+(R_{ia}R_{jt}+R_{ja}R_{it})\delta(p,l)\sigma_{(q,a)(r,s)}
+(R_{it}R_{bl}+R_{jt}R_{bl})\delta(i,p)\sigma_{(q,b)(r,s)}\biggl)\\
=&-\frac{\sqrt{-1}}{2}\sum_{a=1}^n2^{-\delta(i,j)}R_{ia}\bigg(\delta(q,j)\sigma_{(p,a)(r,s)}+\delta(p,j)\sigma_{(q,a)(r,s)})\\
&+R_{ja}(\delta(p,i)\sigma_{(q,a)(r,s)}+\delta(q,i)\sigma_{(p,a)(r,s)}\bigg)\\
&+\frac{\sqrt{-1}}{2}\sum_{a=1}^n\sum_{t=1}^n\sum_{k=1}^m\frac{B}{A}2^{-1-\delta(i,j)}V_{kt}(R_{ia}R_{jt}+R_{ja}R_{it})(V_{kq}\sigma_{(p,a)(r,s)}+V_{kp}\sigma_{(q,a)(r,s)})\\
&-\frac{\sqrt{-1}}{2}\sum_{b=1}^n\sum_{l=1}^n\sum_{t=1}^n\sum_{k=1}^m\frac{B}{A}2^{-1-\delta(i,j)}\biggl(V_{kt}V_{kl}(R_{it}R_{bl}\delta(j,q)+R_{jt}R_{bl}\delta(i,q))\sigma_{(p,b)(r,s)}\\
&+V_{kt}V_{kl}(R_{it}R_{bl}\delta(j,p)+R_{jt}R_{bl}\delta(i,p))\sigma_{(q,b)(r,s)}\biggl )
\end{aligned}
\end{displaymath}

\begin{displaymath}
\begin{aligned}
\sum_{L'\in\Omega'}W_{I\bar{L'}}^2\frac{\partial M_{K\bar{L'}}^2}{\partial Z_{J}}&=\frac{\sqrt{-1}}{2}\frac{B}{A}\sum_{a'=1}^m\sum_{b'=1}^n\sum_{k=1}^n2^{-1-\delta(i,j)}V_{a'k}(R_{ik}R_{jb'}+R_{jk}R_{ib'})\left(\sigma_{(q,b')(r,s)}V_{a'p}+
\sigma_{(p,b')(r,s)}V_{a'q}\right)\\
&=\frac{\sqrt{-1}}{2}\frac{B}{A}\sum_{k=1}^m\sum_{a=1}^n\sum_{t=1}^n 2^{-1-\delta(i,j)}V_{kt}(R_{it}R_{ja}+R_{ja}R_{it})\left(V_{kp}\sigma_{(q,a)(r,s)}+V_{kq}\sigma_{(p,a)(r,s)}\right)
\end{aligned}
\end{displaymath}

So adding this two parts,
\begin{displaymath}
\begin{aligned}
\Gamma_{JI}^K=&-\frac{\sqrt{-1}}{2}\sum_{a=1}^n2^{-\delta(i,j)}\left(R_{ia}(\delta(q,j)\sigma_{(p,a)(r,s)}+\delta(p,j)\sigma_{(q,a)(r,s)})
+R_{ja}(\delta(p,i)\sigma_{(q,a)(r,s)}+\delta(q,i)\sigma_{(p,a)(r,s)})\right)\\
&-\frac{\sqrt{-1}}{2A}\sum_{b=1}^n\sum_{l=1}^n\sum_{t=1}^n\sum_{k=1}^mB2^{-1-\delta(i,j)}\biggl(V_{kt}V_{kl}((R_{it}R_{bl}\delta(j,q)+R_{jt}R_{bl}\delta(i,q))\sigma_{(p,b)(r,s)})\\
&+V_{kt}V_{kl}\Big((R_{it}R_{bl}\delta(j,p)+R_{jt}R_{bl}\delta(i,p)\Big)\sigma_{(q,b)(r,s)}\biggl )
\end{aligned}
\end{displaymath}

If $p=q$, then $\Gamma_{I,J}^K\neq0$ only when $Z_{ij}$ and $Z_{rs}$ belong to the same row or column with $Z_{pp}$.

If $i=r=p$, $\Gamma_{IJ}^K=-\sqrt{-1}R_{js}$-$\frac{\sqrt{-1}B}{2A}\sum_{k=1}^m\sum_{t=1}^n\sum_{l=1}^nV_{kt}V_{kl}R_{jt}R_{sl}$;

if $i=s=p$, $\Gamma_{IJ}^K=-\sqrt{-1}R_{jr}$-$\frac{\sqrt{-1}B}{2A}\sum_{k=1}^m\sum_{t=1}^n\sum_{l=1}^nV_{kt}V_{kl}R_{jt}R_{rl}$;

If $j=r=p$, $\Gamma_{IJ}^K=-\sqrt{-1}R_{is}$-$\frac{\sqrt{-1}B}{2A}\sum_{k=1}^m\sum_{t=1}^n\sum_{l=1}^nV_{kt}V_{kl}R_{it}R_{sl}$;

If $j=s=p$, $\Gamma_{IJ}^K=-\sqrt{-1}R_{ir}$-$\frac{\sqrt{-1}B}{2A}\sum_{k=1}^m\sum_{t=1}^n\sum_{l=1}^nV_{kt}V_{kl}R_{it}R_{rl}$;

If $p<q$, we only consider the case when $Z_{ij}$ belongs to the same row with $Z_{pq}$, $Z_{rs}$ belongs to the same column with $Z_{pq}$, other cases are the same.

Consider the four cases, $i=p= j, r=s=q$; $i=p=j,r<s=q$; $i=p<j, r=s=q$; $i=p=r, j=s=q$ respectively, it is not hard to see that in each of the cases,
we always have $$\Gamma_{IJ}^K=-\frac{\sqrt{-1}}{2}R_{jr}-\frac{\sqrt{-1}B}{4A}\sum_{k=1}^m\sum_{t=1}^n\sum_{l=1}^nV_{kt}V_{kl}R_{jt}R_{rl}.$$
If $i=p<j, r<s=q, i\neq r,$ or $j\neq q$, we have $$\Gamma_{IJ}^K=-\sqrt{-1}R_{jr}-\frac{\sqrt{-1}B}{2A}\sum_{k=1}^m\sum_{t=1}^n\sum_{l=1}^nV_{kt}V_{kl}R_{jt}R_{rl}.$$
The other cases are similar.
Then since we can compute $D(dZ)$ as
\begin{displaymath}
D(dZ_K)=-\sum_{I,J\in\Omega}\Gamma_{I,J}^KdZ_IdZ_J-\sum_{I'\in\Omega',J\in\Omega}\Gamma_{I',J}^KdZ_I'dZ_J-\sum_{I\in\Omega,J'\in\Omega'}\Gamma_{I,J'}^KdZ_IdZ_J'-\sum_{I',J'\in\Omega'}\Gamma_{I',J'}^KdZ_I'dZ_J',\\
\end{displaymath}
we can easily deduce that the $dZdZ$ part in $D(dZ)$ can be written as
$$dZ\left(\sqrt{-1}Y^{-1}+ \frac{\sqrt{-1}B}{2A}Y^{-1}V^tVY^{-1}\right)dZ.$$
The other parts can be got in the similar way and finally we can check that
$$D(dZ)=-\frac{\sqrt{-1}B}{2A}(dZ,dW^t)\left(
     \begin{array}{cc}
      -2\frac{A}{B}Y^{-1}- Y^{-1}V^tVY^{-1} & Y^{-1}V^t \\
       VY^{-1} & -I\\
     \end{array}
   \right)\left(
     \begin{array}{c}
       dZ  \\
       dW \\
     \end{array}
   \right)$$
Thus we have completed the theorem.

\end{document}